\documentclass[preprint,12pt]{elsarticle}

\usepackage[reqno]{amsmath}
\usepackage{amssymb}
\usepackage{amsthm}
\usepackage{mathtools}

\usepackage[english]{babel}
\usepackage{algorithm}
\usepackage{algpseudocode}
\usepackage{tikz-cd}
\usetikzlibrary{calc}

\usepackage[colorlinks=true, linkcolor=blue, citecolor=blue, urlcolor=blue]{hyperref}

\DeclareMathOperator{\tr}{tr}

\DeclareMathOperator{\End}{End}
\DeclareMathOperator{\Aut}{Aut}

\newcommand{\ZZ}{\mathbb{Z}}
\newcommand{\QQ}{\mathbb{Q}}
\newcommand{\FF}{\mathbb{F}}
\newcommand{\OO}{\mathcal{O}}
\newcommand{\Fp}{\mathbb{F}_p}

\newcommand{\Fpbar}{\overline{\mathbb{F}}_p}

\newcommand{\Isog}{\mathcal{I}_{\langle P \rangle}}
\newcommand{\pfrob}{\phi_{p,E}}

\theoremstyle{plain}
\newtheorem{Theorem}{Theorem}[section]
\newtheorem{Lemma}[Theorem]{Lemma}
\newtheorem{Proposition}[Theorem]{Proposition}
\newtheorem{Corollary}[Theorem]{Corollary}

\theoremstyle{definition}
\newtheorem{Definition}[Theorem]{Definition}
\newtheorem{Example}[Theorem]{Example}

\theoremstyle{remark}
\newtheorem{Remark}[Theorem]{Remark}

\begin{document}

\begin{frontmatter}

\title{Supersingular elliptic curves and twisting endomorphisms}

\author[vt]{Sarah Arpin}
\ead{sarpin@vt.edu}

\author[udl]{Josep M. Miret}
\ead{josepmaria.miret@udl.cat}

\author[udl]{Jordi Pujol\`as}
\ead{jordi.pujolas@udl.cat}

\author[udl]{Javier Valera}
\ead{javier.valera.martin@gmail.com}

\address[vt]{Department of Mathematics, Virginia Tech, Blacksburg, VA, USA}
\address[udl]{Departament de Matem\`atica, Universitat de Lleida, Lleida, Spain}

\begin{abstract}
We generalize the notion of twisting endomorphisms, first defined by \cite{_C_P_V_}, to the setting of $\OO$-oriented supersingular elliptic curves. 
We give an algorithm to find supersingular elliptic curves over $\mathbb{F}_p$ with a twisting endomorphism of prime degree $\ell$, and we use it to compute a basis of their full endomorphism rings.
\end{abstract}


\end{frontmatter}

\section{Introduction}\label{Intro}

Isogeny-based cryptography relies on the hardness of several computational problems related to supersingular elliptic curves and their isogenies over finite fields. 
In the heart of this technology lie the path-finding problem  in supersingular $\ell$-isogeny graphs (PFP), the problem of the computation of endomorphism rings of supersingular elliptic curves and  the one endomorphism problem (see \cite{_Ch_G_L_,_K_L_P_T_,_P_W_} for instance). 
A successful example is the SQISign signature scheme by De Feo, Kohel, Leroux, Petit and Wesolowski \cite{_dF_K_L_P_W_, _dF_L_L_W_}, which protects secret keys with the hardness of PFP. 
Recent reductions between these problems show they are all equivalent (see \cite{_W_,_P_W_}).  Then, the knowledge of the full endomorphism ring in the public data is undesirable in practical situations. 

In this paper, we work with the collection of $\OO$-oriented supersingular elliptic curves over $\Fpbar$, for a fixed imaginary quadratic order $\OO$. An $\OO$-oriented supersingular elliptic curve carries with it the information of a (primitive) embedding of $\OO$ into the endomorphism ring of the curve. This allows for navigating the isogeny graph via isogenies coming from the ideal class group of $\OO$, see \cite{_C_K_,_O_,_A_C_L_S_S_T_1,_A_C_L_S_S_T_2}. In this setting, we introduce the notion of $\OO$-oriented twists and $\OO$-twisting endomorphisms, a generalization of the twisting endomorphisms introduced in \cite{_C_P_V_}: 

\begin{Definition}[$\OO$-twisting endomorphism]
    Let $(E,\iota)$ be an $\OO$-oriented supersingular elliptic curve for an imaginary quadratic order $\OO = \ZZ[\omega]$. An endomorphism $\alpha\in\End(E)$ is an \emph{$\OO$-twisting endomorphism} if:
    \[\widehat\iota(\omega) \circ\alpha = \alpha\circ\iota(\omega).\]
\end{Definition}

We show in Theorem~\ref{thm:everyOO} that every $\OO$-oriented supersingular elliptic curve admits an $\OO$-twisting endomorphism. This generalizes the notion of a twisting endomorphism, for which $\omega = \sqrt{-p}$. Following the original nomenclature of \cite{_C_P_V_}, we call such endomorphisms ``twisting endomorphisms'', or ``(Frobenius) twisting endomorphisms'' when there is the need to emphasize.

We give an algorithm to compute supersingular elliptic curves with a (Frobenius) twisting endomorphism of prime degree $\ell$ (see Algorithm~\ref{algotwendo}). 
In certain situations, the quaternion order of endomorphisms generated by Frobenius and the twisting endomorphism can be extended to the maximal quaternion order which is the ring of endomorphisms of the elliptic curve, using the techniques developed in \cite{_F_I_K_M_N_}. We choose $\ell$ and $p$ such that both the factorisation of the classical modular polynomial $\Phi_{\ell}(x,y)$ modulo $p$, as well as the prime factorisation of $(\ell)$ in quadratic orders of discriminant of size around $p$ are successful. We also need to tell if an ideal of such orders is principal or not. The coefficients of $\Phi_{\ell}(x,y)$ for $\ell<1000$ are manageable (see \cite{_B_L_S_}), and both the factorisation of $(\ell)$ and the principality test are fast for small $p$, so none of these are a problem for small sizes. 

However, in real instances endomorphism degrees are often very large and composite. Theorem~\ref{main} gives an if-and-only-if condition for the existence of degree-$n$ twisting endomorphisms. 
It would be interesting to consider extending Algorithm~\ref{algotwendo} using the SuperSolver or WayFinder techniques \cite{SUPERSOLVER, WAYFINDER} to search for degree-$n$ twisting endomorphisms, or potentially even $\OO$-twisting endomorphisms. 

SageMath \cite{sage} code to accompany Algorithm~\ref{algotwendo} is available on GitHub: \url{https://github.com/SarahArpin/twendos}. Our algorithm for finding twisting endomorphisms of degree-$\ell$ does not scale to cryptographic size, since it requires modular polynomials, ideal factorisation, and principality tests in orders of large discriminant.

\subsection{Notation}
Let $ p > 3 $ be a prime number, $\mathbb{F}_q$ a finite field of characteristic $p$, and $\overline{\mathbb{F}}_q$ a fixed algebraic closure. We let 
\begin{align*}
E : y^2 = x^3 + ax +b
\end{align*}
be a supersingular elliptic 
over $\mathbb{F}_q$. The $p$-power Frobenius map
\begin{align*}
\phi_{p,E} : & E(\overline{\mathbb{F}}_q) \longrightarrow  E^{(p)}(\overline{\mathbb{F}}_q)\\
              & (x,y)                             \longmapsto     (x^p,y^p)
\end{align*}
is an isogeny from $E$ to $E^{(p)} : y^2 = x^3 + a^p x +b^p$.
For any point $ P \in E( \overline{\mathbb F}_q )$, we let $\mathcal I_{\langle P \rangle}$ the separable isogeny   
\[
\mathcal I_{\langle P \rangle} \ : \ E \ \longrightarrow \ E / \langle P \rangle
\]
given by Vélu  \cite{_V_}. Let $n = \deg (\mathcal{I}_{\langle P \rangle})$. We call $\Isog$ an $n$-isogeny.

The (geometric) endomorphism ring of $E$, denoted $\End(E)$, 
is isomorphic to a maximal order in a quaternion algebra $B_{p,\infty}/\QQ$.

If $k$ is some extension of  $\mathbb{F}_q$, we write $k$-isomorphisms from $E$ to some other $E' : Y^2 = X^3 + a'X +b'$ with $\rho_u$ for $u\in k$, namely  
\begin{align*}
\rho_{u}(x,y)=(u^2x,u^3y)=(X,Y),\,\,
(u^4 a,u^6 b )=(a',b').
\end{align*}
We say $E$ and $E'$ are isomorphic if they are geometrically isomorphic. We let $\text{Isom}_{\,\mathbb F_p}( E )$ 
be the set of supersingular elliptic curves over $\mathbb F_p$ that are $\mathbb F_p$-isomorphic to $E$.

\subsection{Supersingular elliptic curves over \texorpdfstring{$\Fp$}{Fp}}

In this section, let $E$ be a supersingular elliptic curve defined over $\Fp$.
The $p$-power Frobenius isogeny $\pfrob$ is an endomorphism of $E$. 
We recall the isogeny graph structure theorem in this case, and we refer the reader to \cite{_D_G_} for more detail. 
We denote by $\End_{\mathbb{F}_p}(E)\subset \End(E)$ the subring of endomorphisms which are defined over $\mathbb{F}_p$. For a prime $\ell\neq p$, we can use $\End_{\Fp}(E)$ to determine how many $\Fp$-rational $\ell$-isogenies $E$ has. 
The ring of $\Fp$-endomorphisms $\End_{\Fp}(E)$ of a supersingular elliptic curve $E/\Fp$ necessarily contains the subring $\ZZ[\pfrob]$, which is isomorphic to $\ZZ[\sqrt{-p}]$. 
In fact:
\begin{equation}\label{eq:fpendoring}
    \End_{\Fp}(E)\cong \ZZ[\sqrt{-p}]\text{ or }\ZZ\left[\frac{1 + \sqrt{-p}}{2}\right].
\end{equation}
If $\varphi:E\to E'$ is a prime-degree $\Fp$-rational isogeny and $\End_{\Fp}(E)\not\cong \End_{\Fp}(E')$, then $\deg\varphi = 2$.\footnote{See \cite{_M_M_S_T_V_} for a description of the case for ordinary elliptic curves.} 

\begin{Definition}[Horizontal, ascending, descending isogenies]\label{def:horascdesc}
Suppose $\varphi:E\to E'$ is an isogeny of supersingular elliptic curves defined over $\Fp$. Say $\End_{\Fp}(E)\cong \OO_1$ and $\End_{\Fp}(E')\cong \OO_2$, with $\OO_i\in\{\ZZ[\sqrt{-p}], \ZZ[\frac{1 + \sqrt{-p}}{2}]\}$. 
\begin{itemize}
    \item If $\OO_1 = \OO_2$, $\varphi$ is called \emph{horizontal.}
    \item If $\OO_1 \supsetneq \OO_2$, $\varphi$ is called \emph{descending}.
    \item If $\OO_1 \subsetneq \OO_2$, $\varphi$ is called \emph{ascending}.
\end{itemize}
If the degree of $\varphi$ is coprime to $p$ and $\OO_1\neq\OO_2$, then the smaller of the two orders has index-$\deg\varphi$ in the larger order.
\end{Definition}

\begin{Lemma}\label{lem:Fpisogenies}
The number of $\Fp$-rational $\ell$-isogenies is determined by the factorisation of $x^2 + p$ modulo $\ell$.  

If $\ell\not\in\{2,p\}$:
 \begin{enumerate}
     \item If $\left(\frac{-p}{\ell}\right) = 1$, then each supersingular elliptic curve $E/\Fp$ has exactly two $\Fp$-rational $\ell$-isogenies.

     \item If $\left(\frac{-p}{\ell}\right) = -1$, then there are no $\Fp$-rational $\ell$-isogenies of supersingular elliptic curves over $\Fp$.
 \end{enumerate}

 If $\ell = 2$, the value of the Kronecker symbol $\left(\frac{-p}{2}\right)$ is determined by the equivalence class of $p\pmod{8}$:
 \begin{enumerate}
     \item If $p\equiv 1\pmod{4}$, each supersingular elliptic curve has one $\Fp$-rational 2-isogeny. 

     \item If $p\equiv 3\pmod{4}$, each supersingular elliptic curve with $\End_{\Fp}(E)\cong\ZZ\left[\frac{1 + \sqrt{-p}}{2}\right]$ has three $\Fp$-rational $2$-isogenies and each with $\End_{\Fp}(E)\cong\ZZ[\sqrt{-p}]$ has one $\Fp$-rational $2$-isogeny.
 \end{enumerate}
\end{Lemma}
\begin{proof}
The proof follows from studying the action of the $p$-power Frobenius on $E[\ell]$ in each of the following cases. A fixed linear subspace of $E[\ell]$ corresponds to the kernel of a cyclic $\ell$-isogeny which is defined over $\Fp$. 
For $\ell\neq 2$, the splitting behaviour of $(\ell)$ as an ideal of $\ZZ[\frac{1 + \sqrt{-p}}{2}]$ is the same as it is in $\ZZ[\sqrt{-p}]$, and $(\ell)$ is never ramified. For $\ell = 2$, we split into two cases based on the possibilities for the ring of integers of $K\coloneqq\QQ(\sqrt{-p})$. If $\OO_K = \ZZ[\sqrt{-p}]$, then $(2)$ is ramified. If $\OO_K = \ZZ[\frac{1 + \sqrt{-p}}{2}]$, then $(2)$ is either inert or split, and then $(2)$ is the conductor ideal of $\ZZ[\sqrt{-p}]$. 
For proof details, see \cite[Ch. 2]{_D_G_}. 
\end{proof}

Twisting endomorphisms arise from a particular type of edge in $\mathcal{G}_{\ell}(\mathbb{F}_{p})$: an edge between $\Fp$-twists in a folding connected component of the graph $\mathcal{G}_{\ell}(\mathbb{F}_{p})$.

\subsection{Twisting endomorphisms}
The full ring of endomorphisms of a supersingular elliptic curve over $\Fpbar$ is isomorphic to a maximal order in a quaternion algebra $B_{p,\infty}$ ramified precisely at $p$ and $\infty$, and $\End_{\Fp}(E)$ embeds into this quaternion order. The quaternion algebra $B_{p,\infty}$ is unique up to isomorphism, and (for $p>3$) can be represented as the algebra generated by $1,i,j,ij$ over $\QQ$ according to the relations $ij = -ji$, $j^2 = -p$ and 
\begin{equation}\label{eq:iinBpinfty}
i^2 = \begin{cases}
    -1 &\text{ if }p\equiv 3\pmod{4}\\
    -2 &\text{ if }p\equiv 5\pmod{8}\\
    -q &\text{ if }p\equiv 1\pmod{8},
\end{cases}\end{equation}
where $q\equiv 3\pmod{4}$ and $\left(\frac{p}{q}\right) = -1$. See \cite[Prop. 5.1]{_Pizer_} for a proof.

Since $\ZZ[\sqrt{-p}]\cong\ZZ[\pfrob]\subset\End_{\Fp}(E)$ for every supersingular elliptic curve $E/\Fp$, we can embed $\ZZ[\pfrob]\hookrightarrow\End(E)$ by sending $\pfrob\mapsto j$. This embedding can be extended linearly to an embedding of $\iota_{\pfrob}:\QQ(\sqrt{-p})\hookrightarrow B_{p,\infty}$. 
One choice of embedding $\QQ(\sqrt{-p})\hookrightarrow B_{p,\infty}$ gives rise to a conjugate embedding:
\[\iota_1:\QQ(\sqrt{-p})\hookrightarrow B_{p,\infty},\,\quad\quad \sqrt{-p}\mapsto j; \]
\[\iota_2:\QQ(\sqrt{-p})\hookrightarrow B_{p,\infty},\,\quad\quad \sqrt{-p}\mapsto -j. \]
Since $i^{-1}ji = -j$, these embeddings are conjugate by $i\in B_{p,\infty}^\times$.

A twisting endomorphism of a supersingular elliptic curve $E/\Fp$ of degree $d$ corresponds to an element $\alpha =i(a_2 + a_4j)\in\End(E)$, with $N(\alpha) = d$. Notice:
\[ji(a_2 + a_4j) = -ij(a_2 + a_4j) = -i(a_2 +a_4j)j,\]
or, equivalently in the language of endomorphisms:
\[\pfrob \circ \alpha = -\alpha\circ\pfrob.\]
Twisting endomorphisms were first defined and studied in \cite{_C_P_V_}.
\begin{Definition}\label{def:twistingend}
Let $E$ be a supersingular elliptic curve over $\mathbb{F}_p$. A twisting endomorphism of $E$ is an endomorphism $\alpha\in \End(E)$ such that  
\[
\pfrob \circ \alpha = - \alpha \circ \pfrob .
\]
\end{Definition}

In Section \ref{sec:Otwisting}, we define and study $\OO$-twisting endomorphisms of an $\OO$-oriented supersingular elliptic curve $(E,\iota)$.
In Section \ref{Determination} we assume  $\ell\ne p, \, p>3,\, \ell>2,\,  (\frac{-p}{\ell})=1,\, p\equiv3\pmod{4}$ and obtain a characterisation of the existence of twisting endomorphisms of degree $\ell$ in terms of $\mathbb{F}_p$-rational $\ell$-isogenies. 
In Section \ref{search} we give our algorithm. 
In Section \ref{full} we give an example of a supersingular elliptic curve $E$ over $\mathbb{F}_{439}$ with a twisting endomorphism $\alpha$ of degree $5$ and $j_E\ne 1728$, and we then show how to compute a basis of the full endomorphism ring of $E$.

\section{\texorpdfstring{$\OO$}{O}-twisting endomorphisms}\label{sec:Otwisting}
Supersingular elliptic curves over $\Fp$ are precisely the supersingular elliptic curves over $\Fpbar$ for which the $p$-power Frobenius map is an endomorphism, say $\pfrob$ is the $p$-power Frobenius endomorphism of a supersingular elliptic curve $E/\Fp$. 
Twisting endomorphisms $\alpha$ as above satisfy 
\[\pfrob\circ \alpha = -\alpha\circ\pfrob,\]
as in Definition~\ref{def:twistingend}. To generalize the notion of twisting endomorphism, we first generalize to the setting of orientations. A supersingular elliptic curve $E/\Fp$ with $\pfrob\in\End(E)$ is an elliptic curve with a $\mathbb{Q}(\sqrt{-p})$-\emph{orientation}, which is primitive with respect to either the order $\ZZ[\sqrt{-p}]$ or $\mathcal{O}_K$, see \eqref{eq:fpendoring}. We briefly recall the framework of orientations on supersingular elliptic curves. 
We refer the reader to \cite{_C_K_,_O_} for more details on the general theory of oriented supersingular elliptic curves. 

For the definitions which follow, let $K$ be an imaginary quadratic field in which $p$ is not split, and let $E/\Fpbar$ be a supersingular elliptic curve. 

\begin{Definition}[$K$-orientation]
    A $K$-orientation on $E$ is an embedding
    \[\iota:K\hookrightarrow\End(E)\otimes_\ZZ \QQ,\]
    given by specifying an image of a $\QQ$-generator of $K$ as an element of the quaternion algebra $\End(E)\otimes_\ZZ\QQ$. The orientation $\iota$ is $\OO$-primitive for the imaginary quadratic order $\OO\subseteq K$ such that
    \[\iota(K)\cap\End(E) = \iota(\OO).\]
    As the orientations we consider will be primitive, we will refer to $(E,\iota)$ simply as an $\OO$-oriented supersingular elliptic curve, and we will drop the adjective primitive. 
\end{Definition}

\begin{Definition}[Conjugate orientation]
    For every $\OO$-oriented supersingular elliptic curve $(E,\iota)$, we define a conjugate orientation $\widehat{\iota}$ as follows:
\[\widehat{\iota}(\omega) = \iota(\overline{\omega})\,\text{ for all }\omega\in\OO,\]
where $\overline{\omega}$ denotes the Galois conjugate of the imaginary quadratic element $\omega\in\OO$.
The pair $(E,\widehat{\iota})$ is also an $\OO$-oriented supersingular elliptic curve.
\end{Definition}

\begin{Definition}[$\OO$-oriented isogeny]
    An isogeny $\varphi:E \to E'$ on an $\mathcal{O}$-oriented elliptic curve $(E,\iota)$ induces an orientation $\iota_*$ on the codomain $E'$ as follows:
\[\iota_* ( - ) = \frac{1}{[\deg\varphi]}\varphi\circ\iota(-)\circ\widehat{\varphi}.\]
The resulting pair $(E',\iota_*)$ is a $K$-oriented elliptic curve, but this orientation $\iota_*$ may or may not still be $\mathcal{O}$-primitive. 
\end{Definition}

\begin{Definition}[Isomorphisms of $\OO$-orientations]
Two $\OO$-oriented supersingular elliptic curves over $\Fpbar$ $(E,\iota)$, $(E',\iota')$ are isomorphic if there exists an isomorphism $\eta:E\to E'$ such that
\[\iota'(-) = \eta\circ\iota(-)\circ\eta^{-1}.\]
\end{Definition}

\begin{Example}[Orders of the form {$\ZZ[\sqrt{-d}]$}]
    Let $\OO = \ZZ[\sqrt{-d}]$ for some square-free integer $d\in\ZZ_{>0}$. Suppose $(E,\iota)$ is an $\OO$-oriented supersingular elliptic curve, where $\iota$ is specified by the image $\iota(\sqrt{-d})\in\End(E)$. The conjugate orientation is defined 
    \[\widehat{\iota}:K\hookrightarrow\End(E)\otimes_\ZZ\QQ,\]
    \[\widehat{\iota}(\sqrt{-d}) = \iota(-\sqrt{-d}).\]
    The $\OO$-oriented supersingular elliptic curves $(E,\widehat{\iota})$, $(E,\iota)$ are isomorphic if and only if there exists $\eta\in\Aut(E)$ such that
    \[\widehat{\iota}(\sqrt{-d})\circ\eta = \eta\circ\iota(-\sqrt{-d}) = -\eta\circ\iota(\sqrt{-d}).\]
\end{Example}

\begin{Example}[Elliptic curves over $\Fp$ as oriented elliptic curves]
    Let $E/\Fp$ be a supersingular elliptic curve, and let $\pfrob$ denote the $p$-power Frobenius endomorphism of $E$. The $\Fp$-rational endomorphism ring of $E$ takes one of the following two forms:
    \[\End_{\Fp}(E) \cong \begin{cases}
        \ZZ[\frac{1+\sqrt{-p}}{2}] & \text{ possible only if }p\equiv 3\pmod{4},\\
        \ZZ[\sqrt{-p}] & \text{ possible for any prime $p$. }
    \end{cases}\]
    In particular, $E$ admits a $\QQ(\sqrt{-p})$-orientation, given by 
    \[\iota:\QQ(\sqrt{-p})\to\End(E)\otimes_\ZZ\QQ,\]
    \[\iota(\sqrt{-p}) = \pfrob.\]
    The conjugate of this orientation is a non-isomorphic orientation on $E$:
    \[\widehat{\iota}:\QQ(\sqrt{-p})\to\End(E)\otimes_\ZZ\QQ,\]
    \[\widehat{\iota}(-\sqrt{-p}) = \pfrob.\]
    The orientations $\iota$ and $\widehat{\iota}$ are not equivalent for $j(E)\neq1728$, as there is no automorphism $\eta\in\Aut(E)$ such that: 
    \begin{equation}\label{eq:Fpexample}
        \widehat{\iota}(\sqrt{-p})\circ\eta = \eta \circ\iota(\sqrt{-p}) \Leftrightarrow -\pfrob\circ\eta = \eta\circ\pfrob.
    \end{equation}
    This inspires the consideration of twisting endomorphisms: While it may be difficult to find an automorphism satisfying \eqref{eq:Fpexample}, it may be possible to find an endomorphism satisfying \eqref{eq:Fpexample}. 
\end{Example}

We fix the following notation for the remainder of this section: let $\mathcal{O}$ denote an order in an imaginary quadratic field, say $\mathcal{O} = \ZZ[\omega]$. Let $(E,\iota)$ denote a (primitively) $\mathcal{O}$-oriented supersingular elliptic curve over $\Fpbar$. 
\begin{Definition}[{$\mathcal{O}$-twisting endomorphism}]\label{def:otwistingendo}
    An endomorphism $\alpha\in\End(E)$ is an $\mathcal{O}$-twisting endomorphism if:
    \[\widehat{\iota}(\omega)\circ\alpha = \alpha\circ\iota(\omega).\]
\end{Definition}

To further the argument that $\OO$-oriented supersingular elliptic curves provide a natural generalization of supersingular elliptic curves, we introduce the notion of $\OO$-oriented twists. There are precisely two $\Fp$-isomorphism classes of supersingular elliptic curves over $\Fp$ for each supersingular $j$-invariant $j\in\Fp$, corresponding to the isomorphism classes of twists (quadratic if $j\neq 1728$ or quartic if $j = 1728$). 
Since an orientation $\iota$ also admits a conjugate orientation $\widehat{\iota}$, the $\OO$-oriented curve $(E,\widehat{\iota})$ is a natural ``twist'' to $(E,\iota)$.

\begin{Definition}[{$\mathcal{O}$-oriented twists}]
Let $(E,\iota)$ be an $\mathcal{O}$-oriented supersingular elliptic curve over $\Fpbar$. The $\mathcal{O}$-oriented twist of this curve is defined to be $(E,\widehat{\iota})$.
\end{Definition}

For $\mathcal{O} = \mathbb{Z}[\sqrt{-p}]$, this corresponds to the notion of $\Fp$-twist: one can think of $E,E^t$ as corresponding to the two choices of root of the minimal polynomial $x^2 + p$ of the $p$-power Frobenius map of a supersingular elliptic curve $E/\Fp$. Likewise, $(E,\iota),(E,\widehat{\iota})$ correspond to the two choices of root of the minimal polynomial of the generator $\omega$ of $\mathcal{O} = \ZZ[\omega]$.

\begin{Theorem}
    Let $(E,\iota)$ be an $\OO$-oriented supersingular elliptic curve over $\Fpbar$. Let $n\in\mathbb{Z}_{\geq1}$ be a positive integer. 
    There exists an $\OO$-twisting endomorphism $\alpha$ of $E$ of degree $n$ if and only if there exists a degree-$n$ $\OO$-oriented isogeny $(E,\iota)$ to $(E,\widehat{\iota})$.
\end{Theorem}

\begin{proof}
    Take $\alpha\in\End(E)$ and suppose $\alpha$ is a degree-$n$ twisting endomorphism with respect to the orientation $\iota$. 
    Consider the orientation $\iota_*$ on $E$, induced by the isogeny $\alpha:E \to E$:
    \[\iota_*(-) = \frac{1}{[\deg\alpha]}\alpha\circ\iota(-)\circ\widehat{\alpha}.\]
    As $\iota_*$ is completely determined by the image of $\omega$, we see:
    \begin{equation}
    \begin{split}
        \iota_*(\omega) &= \frac{1}{[\deg\alpha]}\alpha\circ\iota(\omega)\circ\widehat{\alpha}\\
        &=\frac{1}{[\deg\alpha]}\widehat{\iota}(\omega)\circ\alpha\circ\widehat{\alpha}\\
        &=\widehat{\iota}(\omega)
        \end{split}
    \end{equation}
    By definition, $\alpha$ is an isogeny from $(E,\iota)$ to $(E,\widehat{\iota})$. 

    The converse follows quickly from reversing the computation above. Let $\overline{\omega}$ denote the complex conjugate of $\omega$ in $\OO$. If $\alpha:(E,\iota)\to(E,\widehat\iota)$, then:
    \[\widehat\iota(\overline{\omega}) = \frac{1}{[\deg\alpha]}\alpha\circ\iota(\overline{\omega})\circ\widehat{\alpha}.\]
    Post-compose with $\widehat\alpha$ and take the dual of both sides (note $\widehat{\iota(\overline{\omega})}=\iota(\omega)$):
    \[\widehat\iota(\omega)\circ\alpha = \alpha\circ\iota(\omega),\]
    so $\alpha$ is by definition an $\OO$-twisting isogeny.
    
\end{proof}
Twisting endomorphisms and $\OO$-twisting endomorphisms are in general large degree: this comes from the nature of the maximal order $\End(E)$. In Example~\ref{ex:OOtwisting}, we explicitly compute the Deuring correspondence to make the task of finding $\OO$-twisting endomorphisms as easy as possible. 

\begin{Example}[$\OO$-twisting endomorphism]\label{ex:OOtwisting}
    Let $E:y^2 =x^3 + 9x + 19$ be defined over $\FF_{29^2}\coloneqq \FF_{29}[s]/(s^2 +s + 1)$. The endomorphism ring of $E$ is a maximal order in the quaternion algebra $B_{29,\infty}\coloneqq \End(E)\otimes_\ZZ \QQ$, generated by $i,j,ij$ with $i^2 = -2$, $j^2 = -29$, and $ij = -ji$. In particular\footnote{This correspondence is easy to compute because $p$ is so small: one can compute the collection of all maximal orders and find the connecting ideals of reduced norm 2, and line these relationships up with the 2-isogeny graph over $\overline{F}_{29}$.}:
    \begin{equation}\label{eq:p29quat}\End(E)\cong M\coloneqq \ZZ\langle1, i, 1/2 - i/4 + ij/4, -1/2 + i/2 - j/2 \rangle \subset B_{29,\infty}.\end{equation}
    The elliptic curve $E$ has two endomorphisms $\eta,\eta'$ of norm 31 and trace 0, corresponding to the elements $\pm (i + j) \in M$. Without loss of generality, suppose $\eta$ corresponds to $i + j$ and $\eta'$ corresponds to $-i-j$ under the identification in \eqref{eq:p29quat}. Let $K \coloneqq \QQ(r)/(r^2 + 31)$ be the number field of discriminant $-31$, and let $\OO\coloneqq \ZZ[r]$ denote the order in $K$ generated by $r$. The two elements $\eta,\eta'$ are the two choices for $\OO$-orientations on $E$, say:
    \[\iota:K\hookrightarrow\End(E)\otimes_\ZZ \QQ\]
    \[\iota(r) = \eta.\]
    It follows that $\widehat{\iota}(r) = \eta'$. To find an $\OO$-twisting endomorphism of $E$ is to find an endomorphism $\alpha\in\End(E)$ such that:
    \[\eta'\circ\alpha = \alpha\circ\eta.\]
    The endomorphism $\alpha$ will correspond to an element $a$ of $M$ satisfying:
    \[(-i-j)\cdot a = a\cdot (i + j).\]
    Since we have an explicit basis for $M$, we can solve for such an element $a$ explicitly: a $\ZZ$-basis for solutions in $M$ is given by 
    $$\{ij, 29i/4 - j/2 - 27ij/4\}.$$  
    In particular, the element $ij$ of reduced norm $N(ij) = 58$ gives a solution:
    \[(-i - j)ij = -iij -jij = iji + ijj = ij(i + j).\]
    Since $p=29$ divides $N(ij)$, the corresponding isogeny necessarily factors through Frobenius.
    
    \begin{figure}
        \centering
            \begin{tikzpicture}
                \node[] (Eiota) at (0,0) {$(E,\iota)$};
                \node[] (Eiotap) at (3,0) {$(E,\iota^p)$};
                \node[] (Eiotahat) at (6,0) {$(E,\widehat{\iota})$};

                \draw[->] (Eiota) to node[above] {$\alpha_j$} (Eiotap);
                \draw[->] (Eiotap) to node[above] {$\alpha_i$} (Eiotahat);
                \draw[->] (Eiota) to[bend right=20] node[below] {$\alpha$} (Eiotahat);

                \node[] (Eiotaim) at (0,-1.75) {$i+j$};
                \node[] (Eiotapim) at (3,-1.75) {$-i + j$};
                \node[] (Eiotahatim) at (6,-1.75) {$-i-j$};
                \draw[->] (Eiotaim) to (Eiotapim);
                \draw[->] (Eiotapim) to (Eiotahatim);
            \end{tikzpicture}
        \caption{$\OO$-twisting endomorphism diagram to accompany Example~\ref{ex:OOtwisting}. At the top is the factorisation of the $\OO$-twisting endomorphism and the corresponding image orientations. Below, the image $\iota(r)$ of the generator of $\OO$ is given, as an element of the quaternion order $M\cong\End(E)$.}
        \label{fig:OOtwistisog}
    \end{figure}
\end{Example}
In Example~\ref{ex:OOtwisting}, we remark that the order $\OO$ by which $E$ is oriented does not contain a Frobenius element, and so any $\OO$-twisting endomorphism must necessarily factor through Frobenius.  

The benefit of $\OO$-twisting endomorphisms is flexibility: these endomorphisms are not necessarily of elliptic curves defined over $\Fp$, so there are more options. However, this does make the matter of detecting such endomorphisms more difficult. We do not have modular polynomials for $\OO$-oriented elliptic curves, so instead provide a characterisation based on the quaternion side of the picture.

\begin{Theorem}\label{thm:everyOO}
    Every $\OO$-oriented supersingular elliptic curve $E/\Fpbar$ admits an $\OO$-twisting endomorphism.
\end{Theorem}
\begin{proof}
    Let $\iota:\OO\hookrightarrow\End(E)$ be an $\OO$-orientation on $E$ for some imaginary quadratic order $\OO$, say $\OO \cong \ZZ[\omega]$ and $\iota(\omega) = \alpha$. Then, $\widehat{\iota}(\omega) = \widehat{\alpha}$. Set $u\coloneqq \alpha - \frac{Tr(\alpha)}{2}\in \End(E)\otimes_\ZZ\QQ$, where $\End(E)\otimes_\ZZ\QQ$ is a quaternion algebra over $\QQ$ ramified at $p$ and $\infty$ by the Deuring correspondence. An $\OO$-twisting endomorphism $\beta\in\End(E)$ would satisfy
    \[\widehat{\alpha}\circ\beta = \beta\circ\alpha.\]
    Searching for such a $\beta$ in $\End(E)\otimes_\ZZ \QQ$, we are looking for a $\beta$ which satisfies
    \[(\frac{Tr(\alpha)}{2} - u)\circ\beta = \beta\circ(\frac{Tr(\alpha)}{2} + u) \Leftrightarrow-u\circ\beta = \beta\circ u.\]
    Let $K\coloneqq \ker(x\mapsto u\circ x + x \circ u)\subseteq\End(E)\otimes_\ZZ \QQ$. This $K$ is a two-dimensional $\QQ$-vector subspace of $\End(E)\otimes_\ZZ \QQ$. Since $\End(E)$ is a full-rank lattice inside of $\End(E)\otimes_\ZZ \QQ$, it has nontrivial intersection with $K$ and we can choose some $\beta\in \End(E)\cap K$ satisfying $\widehat{\alpha} \circ\beta = \beta\circ\alpha$.
\end{proof}

\begin{Corollary}\label{cor:everyOO}
    Every supersingular elliptic curve $E/\Fpbar$ has at least one pair of non-scalar endomorphisms $\alpha,\beta\in\End(E)$ such that $\alpha$ determines an $\OO$-orientation $\iota:\OO\hookrightarrow\End(E)$ and $\beta$ is an $\OO$-twisting endomorphism. 
\end{Corollary}
\begin{proof}
    The Gross lattice of $\End(E)$ is defined to be the set of endomorphisms $\{2\alpha - Tr(\alpha):\alpha\in\End(E)\}$ is a three-dimensional lattice over $\ZZ$. Let $\alpha$ be a cyclic endomorphism of $E$ in the Gross lattice, where in particular the trace of $\alpha$ is $0$, so $\overline{\alpha} = -\alpha$. Since $\ker\alpha$ is cyclic and $\alpha$ is trace 0, $\alpha$ does not factor through multiplication-by-$m$ for any $m\in\ZZ_{>1}$ and $\alpha$ determines an $\OO$-orientation by the imaginary quadratic order $\OO\coloneqq \ZZ[\alpha]$.
    
    Since $\ZZ[\alpha]$ generates a two-dimensional lattice inside of the four-dimensional $\QQ$-vector space $\End(E)\otimes_\ZZ\QQ$, the kernel of the map $(x\mapsto \alpha \circ x + x\circ \alpha)$ is a two-dimensional $\QQ$-vector subspace of $\End(E)\otimes_\ZZ\QQ$. Since $\End(E)$ is a full-rank lattice in $\End(E)\otimes_\ZZ \QQ$, $\End(E)\cap \ker(x\mapsto \alpha\circ x + x\circ\alpha)$ contains some nonzero element, say $\beta$. In particular, $-\alpha\circ \beta = \beta\circ\alpha$.
\end{proof}

Given an $\OO$-oriented supersingular elliptic curve $(E,\iota)$, we can characterize $\OO$-twisting endomorphisms as follows:
\begin{Lemma}\label{lem:traceofOtwisting}
    Suppose $(E,\iota)$ is an $\OO$-oriented supersingular elliptic curve, and that $\beta\in\End(E)$ is an $\OO$-twisting endomorphism. Then, the trace of $\beta$ is $0$. 
\end{Lemma}
\begin{proof}
    Suppose $\iota(\omega) = \alpha\in\End(E)$, so that $\widehat{\iota}(\omega) = \widehat{\alpha}$. The following series of computations follows from the fact that $\beta$ is $\OO$-twisting:
    \begin{equation}\label{eq:lemma01}
    \begin{split}
        \widehat{\alpha}\circ\beta &= \beta \circ\alpha\\
        \alpha\circ \widehat{\alpha}\circ\beta &= \alpha\circ\beta \circ\alpha\\
        [\deg\alpha]\circ\beta &= \alpha\circ\beta\circ\alpha
    \end{split}
    \end{equation}
    Taking the dual of both sides of the definition of $\OO$-twisting endomorphism, we likewise obtain:
    \begin{equation}\label{eq:lemma02}
    \begin{split}
        \widehat{\beta}\circ\alpha &= \widehat{\alpha}\circ\widehat{\beta}\\
        \alpha\circ\widehat{\beta}\circ\alpha &= [\deg\alpha]\circ\widehat\beta
    \end{split}
    \end{equation}
    Summing equations \eqref{eq:lemma01} and \eqref{eq:lemma02}, 
    \begin{equation*}
    \begin{split}
        [\deg\alpha]\circ(\beta + \widehat\beta) &= \alpha\circ\beta \circ \alpha + \alpha\circ\widehat\beta\circ\alpha \\
        [\deg\alpha]\circ tr(\beta) &= \alpha\circ\alpha\circ tr(\beta),
    \end{split}
    \end{equation*}
    which implies $\tr(\beta) = 0$ since $\alpha\circ\alpha\neq [\deg\alpha]$.
\end{proof}

The relationship between the endomorphisms $\iota(\OO)\subseteq \End(E)$ and the $\OO$-twisting endomorphisms of $E$ is clear from the definitions, but we record these nice properties in the following proposition:
\begin{Proposition}
    Let $(E,\iota)$ be an $\OO$-oriented supersingular elliptic curve. Let $\alpha\in\iota(\OO)\subseteq\End(E)$ and let $\beta_1,\beta_2$ be $\OO$-twisting endomorphisms of $E$. Let $a,b\in\ZZ$.
    \begin{enumerate}[(i)]
        \item $[a]\beta_1 + [b]\beta_2$ is an $\OO$-twisting endomorphism of $E$.

        \item $\beta_i\circ\alpha$ and $\alpha\circ\beta_i$ are $\OO$-twisting endomorphisms of $E$, for any $i\in\{1,2\}$.

    \end{enumerate}
\end{Proposition}

    Any non-scalar endomorphism $\beta\in\End(E)$ determines an imaginary quadratic order $\ZZ[\beta]$ which embeds into $\End(E)$. If $\beta$ is an $\OO$-twisting endomorphism for the $\OO$-oriented supersingular elliptic curve $(E,\iota)$ with $\OO = \ZZ[\omega]$, one may wonder when $\iota(\omega)$ is an $\ZZ[\beta]$-twisting endomorphism. From similar computations to Lemma~\ref{lem:traceofOtwisting}, we see that this is the case if and only if the trace of $\alpha$ is 0. In this case, the composition of two $\OO$-twisting endomorphisms is an endomorphism in $\iota(\OO)$:
    \[\widehat\iota(\beta)\circ(\alpha_1\circ\alpha_2) = \alpha_1\circ\iota(\beta)\circ\alpha_2 = \alpha_1\circ\alpha_2\circ\widehat\iota(\beta).\]
    This gives $\End(E)$ a $\ZZ/2\ZZ$-graded structure.

\section{Detection of twisting endomorphisms}\label{Determination}

\textbf{Standing assumptions for Sections~\ref{Determination} and \ref{search}:} $p>3$ a prime with $p\equiv3\pmod{4}$; $\ell>2$ a prime satisfying $\left(\frac{-p}{\ell}\right) = 1$.

In this section we give a characterisation of (Frobenius) twisting endomorphisms in terms of isogenies (Theorem \ref{main}). 
We have simple models for the $\Fp$-isomorphism classes supersingular elliptic curves.

\subsection{The \texorpdfstring{$j=1728$}{j=1728} scenario}\label{E12}
Let $i \in\mathbb{F}_{p^2}\setminus\Fp$ such that $i^2=-1$. In this section, we show that twisting endomorphisms for $E$ with $j(E)=1728$ do not correspond to isogenies from $E$ to its $\Fp$-twist, but to $\Fp$-rational isogenies post-composed with an automorphism of $E$. 
Take the following representatives of the two $\Fp$-isomorphism classes of supersingular elliptic curves over $\Fp$ with $j$-invariant $1728$:

\begin{equation}\label{E1E2}
\begin{array}{l}
 E_1:y^2=x^3-4x\\
 E_2:y^2=x^3+x
\end{array} .
\end{equation}
The isomorphism
$$\eta:E_1\to E_2$$
$$\eta(x,y)=(-2ix,(-2i-2)y)$$ 
is not defined over $\Fp$, but $\eta$ is defined over $\FF_{p^2}$. 
The isomorphism $\eta$ is not a twisting endomorphism, which can be computed directly as $i^p = -i$ in $\FF_{p^2}$.

\begin{Example}[Twisting endomorphisms for $j = 1728$]\label{ex:1728twendos}
    Let $E/\Fpbar\coloneqq E_1:y^2 =x^3-4x$ as above. The automorphism group of $E$ is generated by the order-4 automorphism $[i]:(x,y)\mapsto(-x,iy)$. 
    The endomorphism ring of $E$ is a maximal order in the quaternion algebra $\End(E)\otimes_\ZZ \QQ$:
    \[\End(E) = \ZZ\langle [1],[i],\frac{[1] + \pfrob}{2}, \frac{[i] + [i]\circ\pfrob}{2}\rangle.\]
    Twisting endomorphisms of $E$ are elements of $\End(E)$ which anti-commute with $\pfrob$. In particular, 
    \[-(a[i] + b\frac{[i] + [i]\circ\pfrob}{2})\circ\pfrob = \pfrob\circ(a[i] + b\frac{[i] + [i]\circ\pfrob}{2}),\]
    where $a$ and $b$ denote integer scalar multiplication maps. 
    Thus, $E$ has a twisting endomorphism of degree-$\ell$ for every $\ell$ which is represented by the norm-form
    \[N(a[i] + b\frac{[i] + [i]\circ\pfrob}{2}) = \ell\]
    \[(a + \frac{b}{2})^2 +b^2\frac{p + 1}{4} = \ell.\]
    Notice that $(a[i] + b\frac{[i] + [i]\circ\pfrob}{2}) = [i] \circ (a + b\frac{[1] + \pfrob}{2})$ and $N(a[i] + b\frac{[i] + [i]\circ\pfrob}{2}) = N(a + b\frac{[1] + \pfrob}{2})$. 
    This last norm is the norm form of the imaginary quadratic order $\ZZ[\frac{1 + \sqrt{-p}}{2}]$. 
    In summary, a supersingular elliptic curve $E$ with $j(E)=1728$ admits a degree-$\ell$ twisting endomorphism for every $\ell$ such that there exists a principal ideal of norm $\ell$ in $\ZZ[\frac{1 + \sqrt{-p}}{2}]$, by post-composing such an endomorphism with the automorphism $[i]$. 
\end{Example}
Example~\ref{ex:1728twendos} serves as a warning: the property of being a twisting endomorphism is not preserved by post-composition with a non-scalar automorphism. Endomorphisms $\eta,[i]\circ\eta\in\End(E)$ with $j(E) = 1728$ have the same kernel, but at most one of these endomorphisms is a twisting endomorphism. 

\subsection{Full supersingular \texorpdfstring{$\ell$}{l}-isogeny graph \texorpdfstring{$ \mathcal{G}_{\ell}(\Fpbar) $}{} }

In this section we describe the relationship between twisting endomorphisms and edges in the $\mathbb{F}_p$-rational supersingular isogeny graph. We recall first how the full supersingular graph is defined.

The vertices of the graph $ \mathcal{G}_{\ell}(\overline{\mathbb{F}}_q) $ are the $\Fpbar$-isomorphism classes of supersingular elliptic curves, labeled by $j$-invariants: 
$$ V_{\overline{\mathbb{F}}_q}=\{ j_E\in \mathbb{F}_{p^2} \mid  E \text{ supersingular elliptic curve in characteristic } p \}$$ 
 The edges of $ \mathcal{G}_{\ell}(\overline{\mathbb{F}}_q) $ are the set 
 $$A_{\overline{\mathbb{F}}_q}=\{ \ell\text{-isogenies } \varphi  : E_1\rightarrow E_2 \mid j_{E_1},j_{E_2}\in V_{\overline{\mathbb{F}}_q}\}$$
of isogenies of degree $\ell$ (defined over $\overline{\mathbb{F}}_{q}$) up to post-composition with an automorphism, between curves in the classes in $V_{\,\overline{\mathbb{F}}_q}$. 

Let $ \Phi_\ell( x , y ) $ be the classical modular polynomial of level $ \ell $ (see \cite[III.8]{_B_S_S_}). Two vertices in $ \mathcal{G}_{\ell}(\overline{\mathbb{F}}_q) $ with $j$-invariants $j_1,j_2$ are adjacent if and only if 
\begin{equation}\label{modpoint}
\Phi_\ell( j_1 , j_2 ) \equiv 0 \pmod p.
\end{equation} 
Lemma~\ref{lem:Fpisogenies} describes precisely when a solution to such an equation exists over $\Fp$.

Let  $ G_1 $, $ \ldots $, $ G_{\ell + 1} $ denote the $ \ell + 1 $ subgroups of order $ \ell $ in $ E( \overline{\mathbb F}_p ) $ and let $E/G_i$  the $ \ell + 1 $ elliptic curves adjacent to $E$ with the $\ell$-isogenies $\mathcal I_{G_i} $. 
Then, the $j$-invariants $$ j( E / G_1 ) ,  \ldots ,  j( E / G_{\ell + 1} ) $$ are the roots of the degree $\ell+1$ polynomial $ \Phi_\ell(  x, j( E )) \equiv 0 \pmod p$.
 
If $ \Phi_\ell( j , j )\equiv 0 \pmod{p}$, then $j$ represents a $\overline{\mathbb{F}}_p$-isomorphism class with an elliptic curve $E$ in characteristic $p$ and an $\ell$-isogeny to another elliptic curve in the same $\overline{\mathbb{F}}_p$-isomorphism class. In particular, if 
$E$ has a twisting endomorphism, then $ \Phi_\ell( x , x ) \pmod{p}$ has a zero (not in $\mathbb{F}_p$), and this appears as a loop in  $ \mathcal{G}_{\ell}(\overline{\mathbb{F}}_q) $. However, the converse statement does not hold.

Proposition~\ref{prop:factpattern} below tells when a zero of $ \Phi_\ell( x , x ) \pmod{p}$, for $\ell>2$, is the $j$-invariant of a \emph{supersingular} elliptic curve. 
For the proof of Proposition~\ref{prop:factpattern} we need some control on the number of edges in $A_{\overline{\mathbb{F}}_q}$ not defined over $\mathbb{F}_p$.
Let $E, C$ supersingular elliptic curves over $ \mathbb F_p$, let 
$$\mathcal{Q}(E)= \{ \mathcal I_{\langle P \rangle} : E \longrightarrow E / \langle P \rangle  \mid \pfrob(P) \notin \langle P \rangle\}$$ 
the set of isogenies from $E$ not defined over $\mathbb{F}_p$ and let 
$$\mathcal{T}(E,C)=\{\mathcal{I} : E \rightarrow E' \mid  \mathcal{I} \in   \mathcal{Q}(E) ,\,  E'  \in \text{Isom}_{\,\mathbb{F}_p}(C)\}.$$

The following Lemma does not require $p\equiv 3\pmod{4}$.

\begin{Lemma} \label{P-L-1}\label{Teven}
For any pair $E, C$ of supersingular elliptic curves defined over $\mathbb{F}_p$ which are not isomorphic over $\Fp$, the  cardinality of $\mathcal{T}(E,C)$ is even.
\end{Lemma}
\begin{proof}
In slightly different terms, this is \cite[Lemma 3.14]{_A_C-N_L_L_N_S_S_}. 
We have $\pfrob \in \End_{\,\mathbb{F}_p} (E) $ and $ E^{(p)} = E $ because $ E $ is defined over $\mathbb F_p $. 
Let $ P_1 \in E( \overline{\mathbb F}_p ) $ have order $ \ell $ and suppose  $E_1=  E / \langle P_1 \rangle  \in \text{Isom}_{\Fp}(C)$. 
Assume the isogeny $ \mathcal I_{\langle P_1 \rangle} : E \to E_1$ is not defined over $\mathbb{F}_p$, so that  $\pfrob(P_1)\notin \langle P_1 \rangle$. 
We thus have $\mathcal I_{\langle P_1 \rangle} \in \mathcal{T}(E,C)$. We have to find another isogeny in $ \mathcal{T}(E,C)$ different from $\mathcal I_{\langle P_1 \rangle}$.

Let  $ P_2 = \pfrob( P_1 ) $. The order of $P_2$ is $\ell$ because $\pfrob$ commutes with the multiplication by $\ell$ map.  The condition that $\mathcal I_{\langle P_1 \rangle} $ is not defined over $\mathbb{F}_p$ is then $P_2\notin \langle P_1 \rangle$. Let $ E_2 = E / \langle P_2 \rangle $. 
The curves $ E_1 , E_2 $ are not $\mathbb{F}_p$-isomorphic, but we have $E_2 =  E_1^p $ by diagram \eqref{di1}. 
At the same time $E_1\in \text{Isom}_{\Fp}(C)$, and $C$ is defined over $\mathbb{F}_p$ by assumption.
Therefore $ j_{E_1}= j_C \in \mathbb F_p $,
and then $$ j_{E_2 } = j_{E_1^{(p)}} = j_{E_1}^p = j_{E_1} .$$ 
Therefore $ E_1 $ and $ E_2 $ are isomorphic and $E_2 \in \text{Isom}_{\Fp}(C)$ follows. 
Moreover, the Frobenius endomorphism satisfies $ \phi_{p,E_2}\circ\phi_{p,E_1} = - [ p ]_{E_1} $ because $ E_1 $ is supersingular over $\Fp$. 
Hence  $ \phi_{p,E}( P_2 ) = \pfrob^2( P_1 ) \in \langle P_1 \rangle \nsubseteq \langle P_2 \rangle$. 
Therefore 
$\mathcal I_{\langle P_2 \rangle} : E \to E_2$  is a non-$\Fp$-rational $\ell$-isogeny different from $\mathcal I_{\langle P_1 \rangle}$ and belongs to $\mathcal{T}(E,C)$ too.
Notice $ E / \langle \pfrob( P_2 ) \rangle = E_1 $ so the argument produces new isogenies only in pairs.
\end{proof}

\begin{Proposition}\label{prop:factpattern}
Let $p>3$ be a prime with $p\equiv3\pmod{4}$; $\ell>2$ a prime satisfying $\left(\frac{-p}{\ell}\right) = 1$. Let $E$ be an elliptic curve over $\mathbb{F}_p$ such that $j_E\ne 0,1728$ is a zero of $\Phi_\ell( X, X)\mod p$ of odd multiplicity. Then  $E$ is supersingular if and only if $ \Phi_\ell( X, j_E)$ neither splits completely into distinct linear factors over $\mathbb{F}_p[X]$ nor has an irreducible factor of degree $\ell$ over $\mathbb{F}_p[X]$.
\end{Proposition}
\begin{proof}
Write $\Phi_\ell( X, j_E)=(X-j_E)f(X)$ with $(X-j_E)\nmid f(X)$. Assume $E/\Fp$ is ordinary. The number of $\mathbb{F}_p$-rational horizontal $\ell$-isogenies in the $\ell$-isogeny volcano component containing $E$ is $$\left(\frac{d_K}{\ell}\right) + 1,$$ where $d_K$ is the discriminant of the maximal order of the quadratic field $K\cong \End(E)\otimes_\ZZ\QQ$, and $d_{\pfrob}=t^2-4q=g^2d_{K}$ where $g$ is the conductor of $\mathbb{Z}[\pfrob]$ in $\mathcal{O}_K$ and $t$ is the trace of $\pfrob$.  
Since $E$ is ordinary and $E$ has an $\ell$-isogeny to itself, $E$ lies in the crater. Since $j_E$ has odd multiplicity as a root of $\Phi_\ell(X,X)$, this multiplicity must be one and therefore $\left(\frac{d_K}{\ell}\right) + 1=1$. Then $$d_K\equiv 0 \pmod{\ell},$$ and hence also $d_{\pfrob}\equiv 0 \pmod{\ell}$. 
By \cite[Proposition VII.2]{_B_S_S_}, if this is the case then there are just two possibilities for $f(X)$: either it  is irreducible of degree $\ell$, or else it splits into $\ell$ linear factors. 
We still have to show these are all different. 
But by the class number formula in \cite[Thm. 7.24]{_C_}
for orders $\mathcal{O}'\subseteq \mathcal{O}$, we deduce $h(\mathcal{O}')=\ell h(\mathcal{O})$ because $j_E\ne 0,1728$. This shows there are exactly $\ell$ descending $\mathbb{F}_p$-rational $\ell$-isogenies from $E$, and this means $ \Phi_\ell( X, j_E)$ splits completely over $\mathbb{F}_p$.

Assume now $E/\Fp$ is supersingular. Since $\left(\frac{-p}{\ell}\right)=1$, then $E$ has precisely two $\mathbb{F}_p$-rational $\ell$-isogenies, and one of these is to itself or its twist. Suppose the other is to $E'/\Fp$.
As above, let $$f( X ) = ( X - j_{E'} )^d g( X )$$ with  $j_{E'} \in \mathbb{F}_p$, $d \ge 1$ and  $(X-j_{E'})\nmid g( X )$. Already we see that $\Phi_\ell(X,j_E)$ cannot have an irreducible factor of degree-$\ell$, since it is degree $\ell+1$ and divisible by both $(X-j_E)$ and $(X-j_{E'})$.
At once, if E is supersingular then $f( X )$ is not irreducible in $\mathbb{F}_p[ X ]$. Furthermore, if $d > 1$ then even if $f( X )$ splits into linear factors in $\Fp[ X ]$, clearly not all of them are different.
Else, if $d = 1$ and $g(X)$ splits into linear factors, then by Lemma \ref{P-L-1} above the multiplicity of each of them
is even. Hence in this case too the $\ell$ linear factors of $f(X)$ are not all different. Finally, if $d = 1$ and $g( X )$ has an irreducible factor in  $\mathbb{F}_p[ X ]$ of degree $r>1$, $r<\ell$, then $ \Phi_\ell( X, j_E)$ cannot split completely nor have an irreducible factor of degree $\ell$.
\end{proof}

Proposition~\ref{prop:factpattern} allows us to check if a twisting endomorphism is truly attached to a supersingular $j$-invariant, for $j\not\in\{0,1728\}$. The twisting endomorphisms for the $j$-invariant 1728 are detected differently, as discussed in Section~\ref{E12}, and the $j$-invariant $j=0$ is supersingular for $p\equiv2\pmod{3}$, so it is not necessary to extend such a proposition for $j=0,1728$. 
In cases where the factorisation pattern of the modular polynomial is not readily available, other efficient supersingularity tests may be applied, such as \cite{_Su_}.

\subsection{Twists}
The twists of an elliptic curve $E$ defined over $\mathbb{F}_p$ are the $\mathbb{F}_p$-isomorphism classes of elliptic curves over $\mathbb{F}_p$ that become isomorphic to $E$ over some extension of $\mathbb{F}_p$: 
\begin{align*}
Twist(E/\mathbb{F}_p) = \{ \text{Isom}_{\,\mathbb F_p}( E' )   \mid   j_{E'}= j_E \}.
\end{align*} 
A representative of a non-trivial class in $Twist(E/\mathbb{F}_p)$ is an elliptic curve $E'$ over $\mathbb F_p$ that is isomorphic to $E$ but not $\mathbb F_p$-isomorphic. Complete sets of representatives of $Twist(E/\mathbb{F}_p)$ are given in \cite[Prop. X.5.4]{_S_}.  

Let $u\in \mathbb{F}_p$ and let 
\begin{align}\label{Et}
E^u : y^2 = x^3 + au^2x +bu^3.
\end{align}
If $\left ( \frac{u}{p}\right )= -1$ and  $\omega\in \mathbb{F}_{p^2}$ is such that $\omega^2=u$, then 
 \begin{equation}\label{rhow}
 \begin{array}{rcrc}
 \rho_{\omega}: &E: y^2 = x^3 + ax +b&\longrightarrow&E^u: y^2 = x^3 + au^2x +bu^3\\
 &(x,y)&\longmapsto &(ux,u\omega y)
 \end{array}
 \end{equation}
is an $\mathbb{F}_{p^2}$-isomorphism because  $u=\omega^2, u\omega=\omega^3$ and $\omega^6$ cancels in $E^u$.  The isomorphism 
$ \rho_{\omega}$ and Frobenius do not commute because 
\begin{equation}\label{antic}
\phi_{p,E^u}\circ \rho_{\omega}(x,y) = ((ux)^p, (u\omega y)^p)=(ux^p,-u\omega y^p)=-\rho_{\omega}\circ\pfrob(x,y).
\end{equation} 

In Lemma \ref{EEt} below we show the class in $Twist(E/\mathbb{F}_p)$ the curve $E^u$ belongs to. 

\begin{Lemma}\label{EEt}
Let $u\in \Fp$ such that $\left(\frac{u}{p}\right) = -1$. Let $E,E^u$ be as in \eqref{Et} and let $E_1,E_2$ be as in \eqref{E12}. Then,
\begin{itemize}
    \item If $j(E)\neq 0,1728$, or if $j(E) = 0$ and $p\equiv 2\pmod{3}$, then there are two $\Fp$-isomorphism classes of twists of $E$, given $Twist(E/\Fp) = \{Isom(E^u),Isom(E)\}$.

    \item If $j(E) = 1728$ and $p\equiv 3\pmod{4}$, then there are two $\Fp$-isomorphism classes of twists of $E$, given $Twist(E/\Fp) = \{Isom(E_1),Isom(E_2)\}$.
  
\end{itemize}
\end{Lemma}
\begin{proof}
    See \cite[Prop. X.5.4]{_S_}, noting that when $p\equiv 3\pmod{4}$, the set of quadratic residues modulo $p$ is precisely equal to the set of quartic residues modulo $p$, and when $p\equiv 2\pmod{3}$, the set of quadratic residues modulo $p$ is precisely equal to the set of sextic residues modulo $p$.
\end{proof}

Note that if $p\equiv 3\pmod{4}$, then $\left(\frac{-1}{p}\right) = -1$, so for $j(E)\neq 1728$, one can take $E, E^{-1}$ as representatives of the two $\Fp$-isomorphism classes in $Twist(E/\Fp)$.
\begin{Definition}\label{DEt}
Let $E: y^2 = x^3 + ax +b$.  We define 
$$
\begin{array}{rl}
E^{-1}:&y^2 = x^3 + ax -b, \\
\end{array}
$$
\end{Definition}
\noindent For $j(E)\neq 1728$, we have the following $\mathbb{F}_{p^2}$-isomorphism 
\begin{equation}\label{rt}
 \begin{array}{rcrc}
 \rho_i: & E: y^2 = x^3 + ax +b&\longrightarrow&E^{-1}: y^2 = x^3 + ax -b\\
 &(x,y)&\longmapsto &(-x, i y)
 \end{array}
 \end{equation}
For $j_E=1728$ then $E^{-1}=E$ and the map $\rho_i$ becomes our automorphism $[i]$ in section \ref{E12}. 
 
\begin{Definition}\label{tw}
Let $E:y^2 = x^3 + ax + b$ be a supersingular elliptic curve over $\Fp$.
We define 
$$
E^t=
\begin{cases}
E^{-1} & \text{ if   }  j_E\ne 1728, \\
y^2 = x^3 - ax & \text{ if   }  j_E=1728.
\end{cases}
$$
We call $E^t$ the twist of $E$.
\end{Definition}

If $E$ is supersingular over $\Fp$, then $|E^t(\Fp)| = |E(\Fp)| = p+1$, and by Tate's isogeny theorem there exists an isogeny $\varphi$ defined over $\mathbb{F}_p$ from $E$ to its twist $E^t$ if we allow $\varphi$ to have arbitrary degree.

In Theorem \ref{main} below we use two facts.
The first is the relationship between an isogeny and the $p$-power Frobenius map, in following commutative diagram:
\begin{equation}\label{di1}
\begin{tikzcd}[sep = 1.5cm]
E \arrow[d, "\phi_{p , E}" ']  \arrow[r,"\mathcal I_{\langle P \rangle}"] & E / \langle P \rangle = C \arrow[d, "\phi_{p , C}"] \\
E^{(p)}  \arrow[r, "\mathcal I_{\langle P^{( p )} \rangle}" '] & E^{( p )} / \langle P^{( p )} \rangle = C^{( p )} 
\end{tikzcd}
\end{equation}
The second is a standard decomposition of endomorphisms into isogenies and geometric isomorphisms.
\begin{Lemma}\label{comp}
Any twisting endomorphism can be decomposed as an isogeny defined over $\mathbb{F}_p$ together with an isomorphism.
\end{Lemma}
\begin{proof}
This is \cite[Lemma 12]{_C_P_V_}.
\end{proof}

\begin{Lemma}\label{lem:twistfrob}
Let $p>3$ be a prime, $E$ a supersingular elliptic curve over $\Fp$ with $j(E)\neq 0,1728$. Let $\tau:E\to E^t$ be an isomorphism from $E$ to its twist. Then, 
\[-\phi_{p,E^t}\circ\tau = \tau\circ\pfrob.\]
Furthermore, if $j(E) = 0$ or $j(E) = 1728$, there exists an automorphism $\eta\in\Aut(E)$ such that 
\[-\phi_{p,E^t}\circ\eta\circ\tau = \eta\circ\tau\circ\pfrob.\]
\end{Lemma}
\begin{proof}
    For $j(E)\neq0,1728$, the result follows even more directly: Since $\tau$ is not defined over $\Fp$, $\tau\circ\pfrob\neq \phi_{p,E^t}\circ\tau$. However, since $\tau$ is separable and has trivial kernel, it factors uniquely through $\tau\circ\pfrob$. In particular, there exists a degree-$p$ isogeny $\lambda:E^t\to E^t$ such that $\lambda\circ\tau = \tau\circ\pfrob$. Comparing degrees and using the fact that $\lambda\neq \phi_{p,E^t}$, we must have $\lambda = -\phi_{p,E^t}$.  

    For $j(E) = 0$ or $j(E) = 1728$, there exists an isomorphism from $E$ to its $\Fp$-twist that is a twisting endomorphism, but post-composition with an automorphism changes whether or not such an isomorphism is a twisting endomorphism so we cannot say that \emph{every} such isomorphism from $E$ to its $\Fp$-twist is a twisting endomorphism.
\end{proof}

Scalar automorphisms do not affect the property of being twisting:
\[-\pfrob\circ\alpha = \alpha\circ\pfrob \Leftrightarrow-\pfrob\circ([\pm1]\circ\alpha) = ([\pm1]\circ\alpha)\circ\pfrob.\]
However, the same does not hold for nonscalar automorphisms. For $j(E)=1728$, take $[i]\in\Aut(E)$ of order 4. This automorphism is a twisting endomorphism:
\[\pfrob\circ[i] = -[i]\circ\pfrob.\]
If $\alpha$ is some twisting endomorphism, then $[i]\circ\alpha$ will not be a twisting endomorphism: 
\[-\pfrob\circ\alpha = \alpha\circ\pfrob \Rightarrow -\pfrob\circ([i]\circ\alpha) = [i]\circ\pfrob\circ\alpha = -[i]\circ\alpha\circ\pfrob.\]
We record this dichotomy in Lemma~\ref{lem:twistfrom01728}.

\begin{Lemma}\label{lem:twistfrom01728}
    Let $p>3$ be a prime.
 Let $E$ be a supersingular elliptic curve over $\Fp$ with $j(E) = 1728$. Every endomorphism of $E$ defined over $\Fp$, say $\alpha\in\End_{\Fp}(E)$, gives rise to a twisting endomorphism, namely $[i]\circ\alpha$. Furthermore, every twisting endomorphism $\beta$ corresponds to an $\Fp$-endomorphism, namely $[i]\circ\beta$.

\end{Lemma}
\begin{proof}
    For $j(E) = 1728$, a direct computation confirms that for $\alpha\in\End_{\Fp}(E)$:
    \[[i]\circ\alpha\circ\pfrob = [i]\circ\pfrob\circ\alpha = -\pfrob\circ [i]\circ\alpha.\]
    Furthermore, for every twisting endomorphism $\beta$:
    \[[i]\circ\beta\circ\pfrob = -[i]\circ\pfrob\circ\beta = \pfrob\circ[i]\circ\beta,\]
    so since $[i]\circ\beta$ commutes with the $p$-power Frobenius, it is an $\Fp$-rational endomorphism of $E$.

\end{proof}

For $j(E)=0$, the `extra' automorphisms of $E/\Fpbar$ do not have trace 0, and are thus not twisting endomorphisms. 

We now state our characterisation of twisting endomorphisms.

\begin{Theorem}\label{main}
    Let $p>3$ be a prime and let $n\in\ZZ_{\geq 1}$. Let $E$ be a supersingular elliptic curve over $\Fp$ with $j(E)\neq1728$. Then, there exists a twisting endomorphism $\alpha\in\End(E)$ which is degree-$n$ if and only if there exists an $\Fp$-rational isogeny $\varphi:E\to E'$ of degree-$n$ such that $E'\in\text{Isom}_{\Fp}(E^t)$.
\end{Theorem}
\begin{proof}
    Suppose $j(E)\neq1728$ and $\alpha\in\End(E)$ is a twisting endomorphism of degree-$n$, and let $\tau:E\to E^t$ denote a twisting isomorphism from $E$ to its twist. Such a $\tau$ exists by Lemma~\ref{lem:twistfrob}. From here, this direction follows from \cite[Lemma 13]{_C_P_V_}, and we briefly reproduce the argument here for reference. Using the definition of twisting endomorphism:
    \begin{equation}
    \begin{split}
        \tau\circ\alpha\circ\pfrob &= \tau\circ(-\pfrob)\circ\alpha = \phi_{p,E^t}\circ\tau\circ\alpha.
    \end{split}
    \end{equation}
    Since $\tau\circ\alpha:E\to E^t$ commutes with Frobenius, it is defined over $\Fp$.

    For the reverse direction, let $ I_{G} : E \longrightarrow E / G$ be an isogeny of degree-$n$ defined over $\mathbb{F}_p$ with kernel $G$ such that $E / G \in  \text{Isom}_{\,\mathbb{F}_p}(E^t)$. 
    Without loss of generality, suppose $E / G = E^t$. 
The following diagram commutes by \eqref{antic}:
\begin{equation}\label{dneqss}
\begin{tikzcd}[sep = 1.5cm]
E \arrow[d, "\pfrob"']  \arrow[r,"\mathcal I_{G}"] & E^t  \arrow[d, "\phi_{p,E^t}"] \arrow[r, "\tau^{-1}"] & E \arrow[d, "\pfrob"]  \\
E \arrow[r,"\mathcal I_{G}"'] & E^t  \arrow[r, "-\tau^{-1}"']  &E
\end{tikzcd}
\end{equation}
The endomorphism $$\alpha\coloneqq\tau^{-1}\circ\rho_u\circ\mathcal I_{G}$$ is a twisting endomorphism of $E$. 
\end{proof}

    In the $j(E)=1728$ case (Lemma~\ref{lem:twistfrom01728}), $\Fp$-rational endomorphisms do yield twisting endomorphisms when post-composed with an order-4 nontrivial automorphism of $E$, whereas for $j\neq1728$, $\Fp$-rational isogenies to their non-trivial $\Fp$-twists yield twisting endomorphisms.

\begin{Example}[Degree-$2$ Twisting Endomorphisms]
The only degree-2 twisting endomorphisms occur for elliptic curves with $j$-invariants in $\{8000,1728\}$. 
    Use the factorisation of $\Phi_2(X,X)\in\Fp[X]$, together with the proof of \cite[Cor. 3.28]{_A_C-N_L_L_N_S_S_}. 
    Only curves with $j(E) \in\{1728,-3375,8000\}$ admit degree-2 endomorphisms \cite[Sec. 2.2]{_A_C-N_L_L_N_S_S_}. 
    The $j$-invariant 1728 is handled in Example~\ref{ex:1728twendos}. 
    The $j$-invariant $-3375$ corresponds to an elliptic curve $E$ whose degree-2 endomorphisms are not defined over $\Fp$, and thus cannot have a twisting endomorphism by Theorem~\ref{main}. In fact, this $j$-invariant has complex multiplication by the order $\ZZ[\frac{1 + \sqrt{-7}}{2}]$, so its degree-2 endomorphisms have trace $1$ and thus cannot be twisting endomorphisms by Lemma~\ref{lem:traceofOtwisting}.
    The $j$-invariant 8000 corresponds to an elliptic curve $E$ with a degree-2 endomorphism which is defined over $\Fp$, as it is defined over $\QQ$. The $j$-invariant $8000$ is a supersingular $j$-invariant for $p\equiv 5$ or $7\pmod{8}$ (since this is when $p$ is not split in the field $\QQ(\sqrt{-2})$ with Hilbert class polynomial $X -8000$). 
\end{Example}

\section{The search in \texorpdfstring{$\mathcal{G}_{\ell}(\mathbb{F}_{p})$}{Gl(Fp)} for \texorpdfstring{$\ell>2$}{l>2}}\label{search}

\textbf{Standing assumptions for Sections~\ref{Determination} and \ref{search}:} $p>3$ a prime with $p\equiv3\pmod{4}$; $\ell>2$ a prime satisfying $\left(\frac{-p}{\ell}\right) = 1$.

In this section, we build to an algorithm for finding twisting endomorphisms of prime degree.

By Theorem \ref{main}, the loop in $\mathcal{G}_\ell(\Fpbar)$ induced by a twisting endomorphism appears as an edge in $\mathcal{G}_\ell(\Fp)$. In this section, we search the graph $\mathcal{G}_\ell(\Fp)$ for edges corresponding to twisting endomorphisms.

The vertex set of the supersingular isogeny graph $\mathcal{G}_{\ell}(\mathbb{F}_{p})$ is
$$V_{\mathbb{F}_p}=\{ \text{Isom}_{\,\mathbb F_p}( E ) \mid  E \text{ supersingular elliptic curve defined over } \mathbb{F}_p \}.$$ 
This set is finer than $V_{\overline{\mathbb{F}}_p}$ as it contains $\mathbb{F}_p$-isomorphism classes which are not uniquely labeled by $j$-invariants. By Lemma \ref{EEt}, 
$V_{\mathbb{F}_p}$ is exactly twice as big as $V_{\overline{\mathbb{F}}_p}$.
For each vertex $v_1$ in $\mathcal{G}_{\ell}(\mathbb{F}_{p})$, let $E_1$ be an $\Fp$-isomorphism class representative of $v_1$, and we draw an edge $(v_1,v_2)$ for every $\Fp$-isogeny $E_1\to E_2$, where $E_2$ is an $\Fp$-isomorphism class representative of $v_2$. Let $A_{\Fp}$ denote the set of edges in $\mathcal{G}_{\ell}(\mathbb{F}_{p})$. 

The vertex set $V_{\Fp}$ can be partitioned by the $\Fp$-endomorphism ring. Let $$ \mathcal{O}_{2} = \mathcal{O}_K,\,  \mathcal{O}_{1} = \mathbb Z[ \sqrt{- p} ]. $$
For $p\equiv 1\pmod{4}$, $\mathcal{O}_2 = \mathcal{O}_1$ and for $p\equiv 3\pmod{4}$, $\mathcal{O}_1\neq\mathcal{O}_2$. Edges are characterized as either horizontal, ascending, or descending according to Definition~\ref{def:horascdesc}.

\begin{Proposition}\label{horizontals}
Let $p>3$ be a prime, let $E$ be a supersingular elliptic curve over $\mathbb{F}_p$, and let $\ell>2$ a prime $\ell\ne p$ such that $(\frac{-p}{\ell})=1$. The elliptic curve $E$ has precisely two outgoing $\ell$-isogenies defined over $\Fp$.
\end{Proposition}
\begin{proof}
See \cite[Theorem 2.7]{_D_G_}.
\end{proof}

\begin{Corollary}\label{horizontalscor}
Let $p$ and $\ell$ as in Proposition \ref{horizontals} above, let $E,E'$ supersingular elliptic curves over $\mathbb{F}_p$ and let $\varphi : E\rightarrow E'$ an $\mathbb{F}_p$-rational $\ell$-isogeny. 
Then, $\End_{\Fp}(E)\cong\End_{\Fp}(E')\cong \OO_i$ for some $i\in\{1,2\}$, and $\varphi$ corresponds to a prime ideal of norm $\ell$ of $ \mathcal{O}_i$.
\end{Corollary}
\begin{proof}
Follows immediately from the proof of Proposition \ref{horizontals}, see \cite{_D_G_}.
\end{proof}
 
Vertices of $\mathcal{G}_{\ell}(\mathbb{F}_{p}) $ for $\ell>2$ corresponding to elliptic curves with non-isomorphic $\Fp$-endomorphism rings are not on the same connected component. Let
$$\mathcal{E}\mathcal{\ell}\mathcal{\ell}_p(\mathcal{O}_i)=\{\text{Isom}_{\, \mathbb{F}_p}(E)\mid \End_{\, \mathbb{F}_p}(E)\cong \mathcal{O}_i\}.$$ 
Both class numbers $ h_i \coloneqq \#\mathcal{C}\ell(\mathcal{O}_i), \,i=1,2$ are odd, since the $2$-rank of $\OO_2$ is one (see e.g. \cite[p. 170]{_J_W_}) and the index of the class group of $\OO_1$ in the class group of $\OO_2$ can be computed explicitly by \cite[Thm. 7.24]{_C_}.

Let $$ ( \ell ) = \mathfrak l_i \overline{\mathfrak l}_i \in \mathcal{O}_i$$ be the prime decomposition of the ideal generated by $\ell$ in  $ O_i $ and let $$ n_i = \text{ord}( \mathfrak l_i ) $$ be the order of $\mathfrak l_i $ in the class group $\mathcal{C}\ell(\mathcal{O}_i)$. 
We have $ n_i = 1 $ if and only if $ \mathfrak l_i $ is a principal ideal of $\mathcal{O}_i$. 
The supersingular $\ell$-isogeny graph $\mathcal{G}_{\ell}(\mathbb{F}_{p})$ for $\left(\frac{-p}{\ell}\right) = 1$ is $2$-regular and every vertex lies in a cycle of length $n_i$. 
The subgraph induced by vertices in $\mathcal{E}\ell\ell_p(\OO_i)$ for each $i=1,2$ is comprised of $h_i/n_i$ connected components, where each connected component is a cycle of length $n_i$ (see \cite{_F_}).
The cycle length $ n_i $ is odd because $ n_i \mid h_i $ and $h_i$ is odd.

To find twisting endomorphisms of degree $\ell>2$, we want to find connected components of $\mathcal{G}_\ell(\Fp)$ with adjacent vertices with the same $j$-invariant. 
In the language of \cite{_A_C-N_L_L_N_S_S_}, we are looking for connected components which \emph{fold}.
By \cite[Theorem 3.18]{_A_C-N_L_L_N_S_S_} an isogeny  $E\xrightarrow{\varphi_{\mid \mathbb{F}_p}} E^t$ appears once in each cyclic component of $ \mathcal{G}_{\ell}(\mathbb{F}_q) $ containing a vertex  
$ \text{Isom}_{\mathbb F_p}( C ) $ such that $j_C=1728$. 
By Lemma \ref{EEt} there is one such cycle in $\mathcal{E}\mathcal{\ell}\mathcal{\ell}_p(\mathcal{O}_1)$ and one in $\mathcal{E}\mathcal{\ell}\mathcal{\ell}_p(\mathcal{O}_2)$. 
Accordingly, we call these cycles $$\mathbf{C}^{1728}_{i}, i=1,2.$$ 
By construction, the curves in Section \ref{E12} satisfy $E_1\in \mathbf{C}^{1728}_{1},E_2\in \mathbf{C}^{1728}_{2}$. 
\begin{Remark}[Reason for requiring $p\equiv 3\pmod{4}$]
  To use the cycles $\mathbf{C}^{1728}_i$ to produce twisting endomorphisms, we need 1728 to be a supersingular $j$-invariant, so our algorithm restricts to $p\equiv 3\pmod{4}$ for that reason.  
\end{Remark}

Let  $m_1$ and $m_2$ be the number of $\ell$-isogenies from $E$ to $C$ that are $\mathbb{F}_p$-rational and not $\mathbb{F}_p$-rational respectively. The multiplicity of 
$(j_E,j_C)$ as a root of $\Phi_{\ell}(X,Y)\mod{p}\equiv 0$ is $m_1+m_2$.

\begin{Corollary}\label{numberisog}
Let $p\equiv 3 \mod{4}$ and let $E,C$ be two supersingular elliptic curves over $\mathbb{F}_p$ such that $\text{Isom}_{\,\mathbb{F}_p}(E)$ and $\text{Isom}_{\,\mathbb{F}_p}(C)$ lie in the same connected component in  $ \mathcal{G}_{\ell}(\mathbb{F}_p) $.
\begin{itemize}
\item[i)] If $C$ is adjacent to $E$, then $m_1+m_2$ is odd.
\item[ii)] If $C$ is not adjacent to $E$, then $m_1+m_2$ is even.
\end{itemize}
\end{Corollary}
\begin{proof}
By Proposition \ref{horizontals} there are two horizontal $\ell$-isogenies incident to any vertex in $ \mathcal{G}_{\ell}(\mathbb{F}_p) $. If $\text{Isom}_{\mathbb{F}_p}(E)$, $\text{Isom}_{\mathbb{F}_p}(C)$ lie in the same cycle, then $m_1=1$ in case $i)$ since $\mathcal{G}_\ell(\Fp)$ has no multi-edges, and $m_1=0$ in case $ii)$. In both cases $m_2$ is even by Lemma \ref{Teven}.
\end{proof}

Our algorithm for finding twisting endomorphisms (Algorithm~\ref{algotwendo}) checks the parity of the multiplicity of the roots of $\Phi_{\ell}(X,X) \mod{p}$ in the connected components $\mathbf{C}^{1728}_{i}$. 

\begin{Proposition}\label{oddmult}
Let $p\equiv 3 \mod{4}$, $\ell\ne p$ such that $(\frac{-p}{\ell})=1$ and let $E$ a supersingular elliptic curve over $\mathbb{F}_p$ with a twisting endomorphism and $j_E\ne 1728$. Then, the multiplicity of $j_E$ as a zero of $\Phi_{\ell}(X,X) \mod{p}$ is odd.
\end{Proposition}
\begin{proof}
The twisting endomorphism is a loop in the full supersingular  $\ell$-isogeny graph $ \mathcal{G}_{\ell}(\Fpbar) $ hence a zero of $\Phi_{\ell}(X,X) \mod{p}$. Let $m$ the multiplicity of this zero. By Theorem \ref{main}, the loop in $ \mathcal{G}_{\ell}(\Fpbar) $ corresponds to an isogeny $\varphi_{\mid \mathbb{F}_p}:E\rightarrow E^t$ in $ \mathcal{G}_{\ell}(\mathbb{F}_q) $. Since this edge must be on a folding component, it lies in one of the cycles $\mathbf{C}^{1728}_{i}$. Hence, $\text{Isom}_{\,\mathbb{F}_p}(E^t)$ is adjacent to  $\text{Isom}_{\,\mathbb{F}_p}(E)$  in $C^{1728}_{i}$, and by Corollary~\ref{numberisog} $m$ is odd.  
\end{proof}

By the theory of volcanoes of ordinary elliptic curves, the degree of a vertex in a cycle of horizontal $\ell$-isogenies of ordinary elliptic curves is at most $2$ (see \cite{_K_}).  Therefore a multiplicity at least 3 of a root $j_E$ of $\Phi_{\ell}(X,X) \mod{p}$ implies $j_E$ is the $j$-invariant of a supersingular elliptic curve. The cases of multiplicity $1$ can be both ordinary or supersingular, but Proposition \ref{prop:factpattern} above is enough to distinguish these two cases.

We justify now the main steps of Algorithm~\ref{algotwendo}. By Theorem \ref{main}, a twisting endomorphism $\alpha\in \End(E)$ of degree $\ell$ corresponds to an $\ell$-isogeny $\varphi_{\mid \mathbb{F}_p}:E\rightarrow E^t$ in  $\mathcal{G}_{\ell}(\mathbb{F}_{p})$. 
By Corollary \ref{horizontalscor}, our hypotheses imply $\mathcal{G}_{\ell}(\mathbb{F}_{p})  =\mathcal{E}\mathcal{\ell}\mathcal{\ell}_p(\mathcal{O}_1) \oplus  \mathcal{E}\mathcal{\ell}\mathcal{\ell}_p(\mathcal{O}_2)$ and $\varphi_{\mid \mathbb{F}_p}$ is an ideal of norm $\ell$ in one of the orders $\mathcal{O}_1$ or $\mathcal{O}_2$. 
The isogeny $\varphi_{\mid \mathbb{F}_p}$ is an edge in one of the cycle components $\mathbf{C}_i^{1728}$ of $\mathcal{G}_{\ell}(\mathbb{F}_{p})$ containing a vertex with $j = 1728$. 
Our algorithm first computes if the length of the cycles $\mathbf{C}_{i}^{1728}$ is equal to $1$ or otherwise. 
We find the prime decomposition of $ ( \ell ) = \mathfrak l_i \overline{\mathfrak i}_i \in \mathcal{O}_i$ and we let $n_i$ be the order of $ \mathfrak l_i$ in $\mathcal{C}\ell(\mathcal{O}_i)$. 
If $ \mathfrak l_i$ are both principal then $n_i=1$, $\varphi_{\mid \mathbb{F}_p}$ happens only for $j=1728$ and we are done ($j_E=1728$ is always a root of $\Phi_{\ell}(X,X) \mod{p}$). 
If some $\mathfrak l_i$ is not principal in $\mathcal{O}_i$ then  $n_i>1$, and we will find a new $j_E\ne 1728$ only if $j_E$ is a root of $\Phi_{\ell}(X,X) \mod{p}$ which is a supersingular $j$-invariant. 
By Proposition \ref{oddmult}, if this is the case then the multiplicity $m_{E}$ of $j_{E}$ as a root of  $\Phi_{\ell}(j_E,X)\mod{p} $ is odd. 
If $m_E\ge 3$ we found our $j_E$, and if $m_E=1$ then we found a new $j$ only if the factorisation pattern of  $\Phi_{\ell}(j_E,X)\mod{p} $ matches those in Proposition \ref{prop:factpattern}. 
Notice Proposition \ref{prop:factpattern} requires $j\ne 0$. 
The case $j_E=0$ can only happen if $n_2>1$, and we find $\varphi_{\mid \mathbb{F}_p}$ for $j=0$ only if  $\Phi_{\ell}(X,X) \mod{p}$ has a factor $X^k$ with $k$ odd and $p\equiv 2 \mod 3$.

\begin{algorithm}
\caption{Twisting endomorphisms of degree $\ell>2$}\label{algotwendo}
\begin{algorithmic}[1]
\Require A prime $p\equiv 3\pmod{4}$, a prime $\ell>2$ such that $(\frac{-p}{\ell})=1$.
\Ensure A list of all $j$-invariants  of elliptic curves $E/\Fp$ with a twisting endomorphism of degree $\ell$.
\State $J\gets [\hspace{0.2cm}]$ \textit{the list of $j$-invariants to output}
\State \verb|hasFoldedCycle| $\gets$ \verb|False| 
\State Find $\mathcal{O}_2=\mathcal{O}_{\QQ(\sqrt{-p})}, \mathcal{O}_1=\mathbb{Z}[\sqrt{-p}]$ 
\State Factor $(\ell)\OO_i = \mathfrak{l}_i\bar{\mathfrak{l}}_i$ 
\If{$\mathfrak{l}_i$ is principal in $\OO_i$ for some $i \in\{1,2\}$} 
    \State Add  $j=1728$ to $J$.
\EndIf
\If{$\mathfrak{l}_1$ and $\mathfrak{l}_2$ are principal} 
    \State \Return $J$
\ElsIf{$\mathfrak{l}_1$ is not principal in $O_1$} \textit{$\mathbf{C}_1^{1728}$ has an edge $E\xrightarrow{\varphi_{\mid \mathbb{F}_p}} E^t$}
    \State \verb|hasFoldedCycle| $\gets$ \verb|True|.
\ElsIf{$\mathfrak{l}_2$ is not principal in $O_2$} \textit{$\mathbf{C}_2^{1728}$ has an edge $E\xrightarrow{\varphi_{\mid \mathbb{F}_p}} E^t$}
    \State \verb|hasFoldedCycle| $\gets$ \verb|True|.
\EndIf 
\If{\texttt{hasFoldedCycle}}
    \State $R\gets[(\rho,m):\Phi_\ell(\rho,\rho)\equiv 0\pmod{p}\text{ with odd multiplicity }m]$
    \State Remove pairs $(\rho,m)$ from $R$ with $\rho=1728$.
    \For{$(\rho,m)\in R$}
    \If{$\rho=0$}
        \State Add $\rho$ to $J$ if $p\equiv 2\pmod{3}$.
    \ElsIf{$\Phi_\ell(X,\rho)$ neither splits completely into distinct linear factors nor has an irreducible degree-$\ell$ factor}\label{alg:prop3.3}
        \State Add $\rho$ to $J$.
    \EndIf
    \EndFor 
\EndIf
\State \Return $J$
\end{algorithmic}
\end{algorithm}

\begin{Theorem}\label{thm:algo}
    Let $p\equiv 3\pmod{4}$, $\ell>2$ a prime distinct from $p$, and $\left(\frac{-p}{\ell}\right) = 1$.    
    Then, Algorithm~\ref{algotwendo} returns precisely the set of $j$-invariants $j(E)\in\Fp$ for which supersingular elliptic curves $E/\Fpbar$ with that $j$-invariant admits a twisting endomorphism of degree-$\ell$.
\end{Theorem}
\begin{proof}
    The special case of $j(E) = 1728$ is handled in Example~\ref{ex:1728twendos}, Lemma~\ref{lem:twistfrom01728}.
    Suppose $j(E)\neq1728$. By Theorem~\ref{main}, degree-$\ell$ twisting endomorphisms correspond to $\Fp$-rational isogenies from $E$ to $E^t$.
    Such $\Fp$-rational degree-$\ell$ isogenies between twists is only possible for at most two $j(E)$ with $E$ and $E^t$ on the same connected component of $\mathcal{G}_\ell(\Fp)$ as an isomorphism class with $j$-invariant $1728$. 
    In particular, we are looking for isomorphism classes which are opposite from vertices corresponding to 1728-isomorphism classes on the folding components $\mathbf{C}_1^{1728}$ and $\mathbf{C}_2^{1728}$ of $\mathcal{G}_\ell(\Fp)$. 
    By Proposition~\ref{oddmult}, such isomorphism classes will have $j$-invariants which are odd multiplicity roots of the mod-$p$ reduction of the $\ell$-modular polynomial $\Phi_\ell(X,X)\in\Fp[X]$. 
    The algorithm uses the fact that the number of vertices in the cycle $\mathbf{C}_i^{1728}$ is equal to the order of an ideal above $\ell$ in the class group of $\OO_i$. 
    If these ideals are principal, only 1728 will be represented on these connected components and 1728 will have an $\Fp$-rational endomorphism which corresponds to a twisting endomorphism when post-composed with $[i]$. 
    Otherwise, the number of vertices on the cycle will be an odd integer greater than or equal to three and the cycle will contain an edge that corresponds to a twisting endomorphism.
    Proposition~\ref{prop:factpattern} gives a condition on the factorisation pattern of $\Phi_\ell(X,\rho)$ by which we can recognize the supersingular roots of $\Phi_\ell(X,X)$ (see Algorithm~\ref{algotwendo} line~\ref{alg:prop3.3}). This test fails for $\rho=0,1728$, so $\rho=1728$ is handled above and the supersingularity of $\rho=0$ is checked by determining $p\pmod{3}$.
\end{proof}

\section{Full endomorphism ring computations}\label{full}
As an application, we find a $\mathbb{Z}$-basis of the maximal order $\mathcal{M}=\End(E)$ of a curve  found by our algorithm. 

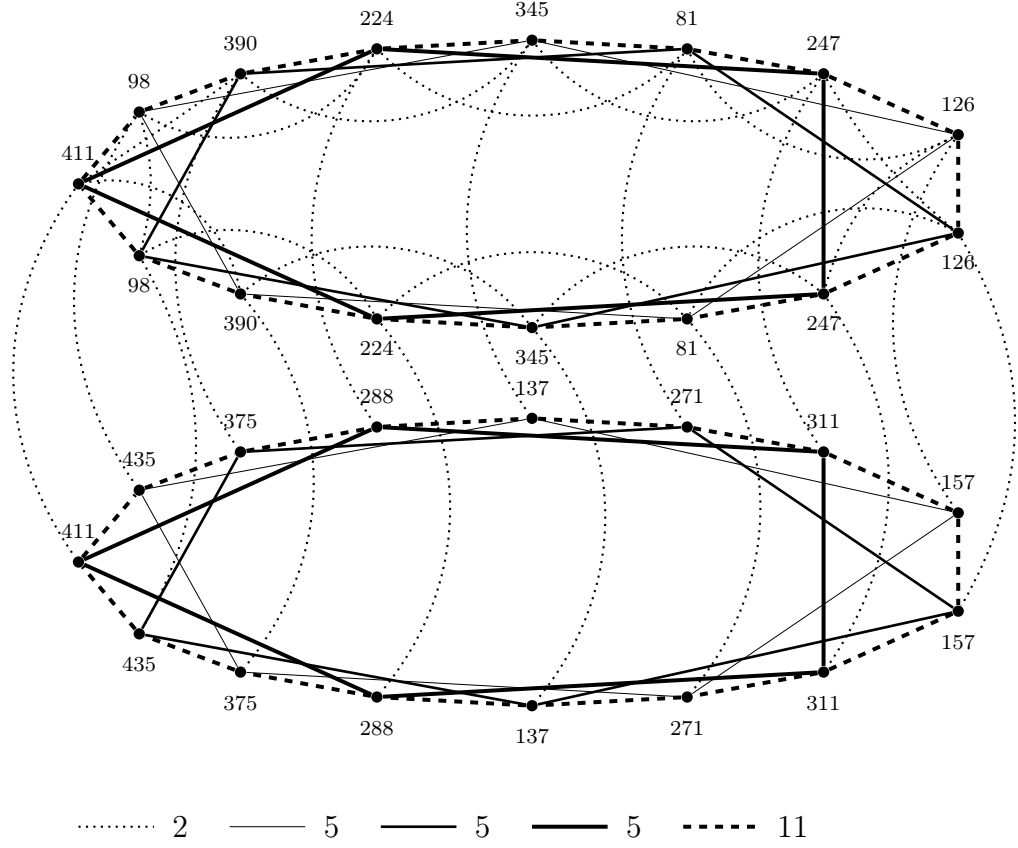
\begin{figure}[ht]
\begin{center}
\begin{tikzpicture}[
scale=1,
node/.style={circle, fill=black, inner sep=1.5pt},
lab/.style={font=\scriptsize},
edgeGreen/.style={line width=1.5pt,dashed},
edgeBlue/.style={line width=1pt},
edgeOrange/.style={line width=0.8pt, dotted},
edgeRed/.style={line width=0.3pt},
edgeBlack/.style={line width=1.5pt},
edgeMain/.style={line width=0.8pt, dotted}
]

\def\a{6}
\def\b{1.9}


\foreach \i/\ang/\lab in {
1/180/411,
2/150/98,
3/130/390,
4/110/224,
5/90/345,
6/70/81,
7/50/247,
8/20/126
}{
\node[node] (T\i) at ({\a*cos(\ang)},{\b*sin(\ang)+2.5}) {};
\node[lab, above =1.7mm] at (T\i) {\lab};
}

\foreach \i/\ang/\lab in {
9/-20/126,
10/-50/247,
11/-70/81,
12/-90/345,
13/-110/224,
14/-130/390,
15/-150/98
}{
\node[node] (T\i) at ({\a*cos(\ang)},{\b*sin(\ang)+2.5}) {};
\node[lab, below =1.5mm] at (T\i) {\lab};
}


\foreach \i/\ang/\lab in {
1/180/411,
2/150/435,
3/130/375,
4/110/288,
5/90/137,
6/70/271,
7/50/311,
8/20/157
}{
\node[node] (B\i) at ({\a*cos(\ang)},{\b*sin(\ang)-2.5}) {};
\node[lab, above =1.7mm] at (B\i) {\lab};
}

\foreach \i/\ang/\lab in {
9/-20/157,
10/-50/311,
11/-70/271,
12/-90/137,
13/-110/288,
14/-130/375,
15/-150/435
}{
\node[node] (B\i) at ({\a*cos(\ang)},{\b*sin(\ang)-2.5}) {};
\node[lab, below =1.7mm] at (B\i) {\lab};
}



\draw[edgeBlack]
(T1)--(T4)--(T7)--(T10)--(T13)--(T1);

\draw[edgeRed]
(T2)--(T5)--(T8)--(T11)--(T14)--(T2);

\draw[edgeBlue]
(T3)--(T6)--(T9)--(T12)--(T15)--(T3);

\draw[edgeGreen]
(T1)--(T2)--(T3)--(T4)--(T5)--(T6)--(T7)--(T8)--(T9)--(T10)--(T11)--(T12)--(T13)--(T14)--(T15)--(T1);


\draw[edgeBlack]
(B1)--(B4)--(B7)--(B10)--(B13)--(B1);

\draw[edgeRed]
(B2)--(B5)--(B8)--(B11)--(B14)--(B2);

\draw[edgeBlue]
(B3)--(B6)--(B9)--(B12)--(B15)--(B3);

\draw[edgeGreen]
(B1)--(B2)--(B3)--(B4)--(B5)--(B6)--(B7)--(B8)--(B9)--(B10)--(B11)--(B12)--(B13)--(B14)--(B15)--(B1);

\draw[edgeOrange] (T1) edge [bend right=35] (B1);
\draw[edgeOrange] (T2) edge [bend right=35] (B2);
\draw[edgeOrange] (T3) edge [bend right=35] (B3);
\draw[edgeOrange] (T4) edge [bend right=35] (B4);
\draw[edgeOrange] (T5) edge [bend right=35] (B5);
\draw[edgeOrange] (T6) edge [bend right=35] (B6);
\draw[edgeOrange] (T7) edge [bend right=35] (B7);
\draw[edgeOrange] (T8) edge [bend right=35] (B8);

\draw[edgeOrange] (T9) edge [bend left=30] (B9);
\draw[edgeOrange] (T10) edge [bend left=35] (B10);
\draw[edgeOrange] (T11) edge [bend left=40] (B11);
\draw[edgeOrange] (T12) edge [bend left=40] (B12);
\draw[edgeOrange] (T13) edge [bend left=40] (B13);
\draw[edgeOrange] (T14) edge [bend left=40] (B14);
\draw[edgeOrange] (T15) edge [bend left=30] (B15);

\draw[edgeOrange] (T1) edge [bend right=1] (T3);
\draw[edgeOrange] (T3) edge [bend right=45] (T5);
\draw[edgeOrange] (T5) edge [bend right=45] (T7);
\draw[edgeOrange] (T7) edge [bend right=10] (T9);
\draw[edgeOrange] (T9) edge [bend right=45] (T11);
\draw[edgeOrange] (T11) edge [bend right=45] (T13);
\draw[edgeOrange] (T13) edge [bend right=45] (T15);
\draw[edgeOrange] (T15) edge [bend right=15] (T2);
\draw[edgeOrange] (T2) edge [bend right=45] (T4);
\draw[edgeOrange] (T4) edge [bend right=45] (T6);
\draw[edgeOrange] (T6) edge [bend right=45] (T8);
\draw[edgeOrange] (T8) edge [bend right=25] (T10);
\draw[edgeOrange] (T10) edge [bend right=45] (T12);
\draw[edgeOrange] (T12) edge [bend right=45] (T14);
\draw[edgeOrange] (T14) edge [bend right=45] (T1);


\begin{scope}[shift={(-6,-8)}]
\draw[edgeOrange] (0,2) -- (1,2);
\node[right] at (1.1,2) {2};

\draw[edgeRed] (2,2) -- (3,2);
\node[right] at (3.1,2) {5};

\draw[edgeBlue] (4,2) -- (5,2);
\node[right] at (5.1,2) {5};

\draw[edgeBlack] (6,2) -- (7,2);
\node[right] at (7.1,2) {5};

\draw[edgeGreen] (8,2) -- (9,2);
\node[right] at (9.1,2) {11};

\end{scope}

\end{tikzpicture}

\end{center}
\caption{Supersingular isogeny graph over $\FF_{439}$ with degree 2,5, and 11 isogenies depicted as edges.}
\label{graph5112}
\end{figure}

\begin{Proposition}
The 
supersingular elliptic curve $E : y^2 = x^3 + 169 x + 307 $ over $\mathbb{F}_{439}$ with $j_E=247$ 
has two twisting endomorphisms $\alpha,\beta$ such that $\alpha^2=-5,\,\beta^2=-22$ and $\End(E)= \langle 1,\alpha,\beta,\alpha\circ\beta\rangle$. 
\end{Proposition}

\begin{proof}
With $p = 439$ and $\ell = 5$ our algorithm finds two elliptic curves
\begin{align*}
E : y^2 = x^3 + 169 x + 307\\
C : y^2 = x^3 + 274 x + 257
\end{align*}
both with $j$-invariant $j_E=j_C=247$. See Figure \ref{graph5112} below. The solid edge joining them is a $5$-isogeny $\mathcal{I}_{\langle P\rangle}:E\to C$ with kernel generated by 
$P=(121 , 155)\in E(\mathbb{F}_{439})$, and composing with an isomorphism $\eta:C\to E$ gives the twisting endomorphism $\alpha$ of degree $5$. 
We want to find another twisting endomorphism $\beta\in \End(E)$ of  degree $m$ such that $\beta^2=-m$ and
$$\End(E)\cong \langle 1,\alpha,\beta,\alpha\circ\beta\rangle.$$ 
We let $t=\tr(\alpha\circ\beta)$ and we impose maximality by solving 
$$
\begin{array}{|cccc|}
tr( [ 1 ] \circ [ 1 ] )& tr( [ 1 ] \circ \alpha ) &tr( [ 1 ] \circ \beta )& tr( [ 1 ] \circ \alpha\circ\beta )\\
tr( \alpha \circ [ 1 ] )& tr( \alpha \circ \alpha ) &tr( \alpha \circ \beta ) &tr( \alpha \circ \alpha\circ\beta )\\
tr( \beta \circ [ 1 ] ) & tr( \beta \circ \alpha ) & tr( \beta \circ \beta ) & tr( \beta \circ \alpha\circ\beta )\\
tr( \alpha\circ\beta \circ [ 1 ] ) & tr( \alpha\circ\beta \circ \alpha ) &tr( \alpha\circ\beta \circ \beta ) &tr( \alpha\circ\beta \circ \alpha\circ\beta )
\end{array}$$
$$=
\begin{array}{|cccc|}
2&0&0&t\\
0&-10&t&0\\
0&t&-2m&0\\
t&0&0&t^2-10m
\end{array}
=-(t^2-20m)^2=-439^2
$$
(see \cite[Thm. 15.5.5.]{_VO_}). The equation $t^2-20m=\pm439$ has many solutions 
$$(t,m)\in\{(1,22),(9,26),(11,28),(19,40),(21,44),(29,64),\ldots\}$$
each corresponding to some endomorphism $\beta\in\End(E)$. 
For example, for $(t,m) = (1,22)$, we can detect a twisting endomorphism of degree-$22$ by finding an $\Fp$-isogeny of degree-22 connecting $E$ with $E^t$. See Figure~\ref{graph5112}: from the 247 at the top of the page, there is an 11-isogeny to 126 (thick, dashed), and from 126 there is a   2-isogeny (thin, dotted) to the second 247. In particular, these isogenies are:
\[\varphi_0:E\to E_{126}:y^2 = x^3 + 392x + 186\,\,,\,\,\deg\varphi_0 = 11\]
\[\varphi_1:E_{126}\to C\,\,,\,\,\deg\varphi_1 = 2\]
\[\beta\coloneqq \eta\circ\varphi_1\circ\varphi_0.\]
The above computation shows that $\End(E)\cong\langle 1,\alpha,\beta,\alpha\circ\beta\rangle$, since this order is maximal. 

\end{proof}

\section{Conclusions and Future Work}
In this paper, we introduced $\OO$-twisting endomorphisms of $\OO$-oriented supersingular elliptic curves in Definition~\ref{def:otwistingendo}. This notion generalizes (Frobenius) twisting endomorphisms introduced by \cite{_C_P_V_}. We proved Theorem~\ref{main} characterizing twisting endomorphisms, and gave Algorithm~\ref{algotwendo} for computing twisting endomorphisms. Whereas only supersingular elliptic curves over $\Fp$ can admit twisting endomorphisms, in Theorem~\ref{thm:everyOO} and Corollary~\ref{cor:everyOO} we proved that all supersingular elliptic curves admit $\OO$-twisting endomorphisms. In future work, we plan to give an algorithm for computing $\OO$-twisting endomorphisms and explore implications on cryptographic hard problems such as the endomorphism ring computation problem and the path-finding problem.

\section*{Acknowledgements}
S. Arpin thanks the Centre de Recerca Matem\`atica (CRM, Barcelona) for support to visit Barcelona during which the connections for the start of this collaboration formed.
S. Arpin thanks the Cryptography \& Graphs Research Group at Universitat de Lleida for their support to visit Lleida. 
S. Arpin was supported in part by the Commonwealth Cyber Initiative, an investment in the advancement of cyber R\&D, innovation, and workforce development. For more information about CCI, visit \url{www.cyberinitiative.org}.

The last three authors were supported in part by the R\&D+i
project PID2021-124613OB-I00 funded by MICIU/AEI/10.13039/501100011033
and ERDF/EU.


\end{document}